\documentclass[11pt]{amsart}

\usepackage{amsmath, amssymb, amsthm, mathrsfs, graphicx, latexsym}
\usepackage[inline]{enumitem}
\usepackage{xspace, ifthen} % comment}
\usepackage[all]{xy}
\usepackage[T1]{fontenc}
\usepackage[utf8]{inputenc}
\usepackage{lmodern}
\usepackage[colorlinks=true,linkcolor=blue,urlcolor=blue, citecolor=blue]%
{hyperref}
\usepackage{mathptmx}
\usepackage{microtype}
\usepackage{framed}
\usepackage{graphicx}
\usepackage[title]{appendix}

\usepackage[centering, includeheadfoot, hmargin=1in, vmargin=1in,
headheight=30.4pt]{geometry}

\newtheorem{theorem}{Theorem}[section]
\newtheorem{lemma}[theorem]{Lemma}
\newtheorem{proposition}[theorem]{Proposition}
\newtheorem{corollary}[theorem]{Corollary}

\numberwithin{equation}{theorem}

\theoremstyle{plain}
\newtheorem*{theoremA}{Theorem A}
\newtheorem*{theoremB}{Theorem B}
\newtheorem*{theoremC}{Theorem C}

\theoremstyle{definition}
\newtheorem{dfn}[theorem]{Definition}
\newtheorem{definition}[theorem]{Definition}

\newtheorem{defi}[theorem]{Definition}

\newtheorem*{dfn-plain}{Definition}

\theoremstyle{remark}

\newtheorem{ntn}[theorem]{Notation}
\newtheorem{rem}[theorem]{Remark}

\newtheorem{remark}[theorem]{Remark}

\newtheorem{exm}[theorem]{Example}

\newtheorem{convention}[theorem]{Convention}
\newtheorem{example}[theorem]{Example}
\newtheorem*{rem-plain}{Remark}

\DeclareMathOperator{\reg}{reg}

\DeclareMathOperator{\Spec}{Spec}
\DeclareMathOperator{\trdeg}{trdeg}

\DeclareMathOperator{\Pic}{Pic}
\DeclareMathOperator{\codim}{codim}
\DeclareMathOperator{\rank}{rank}

\DeclareMathOperator{\ord}{ord}
\DeclareMathOperator{\lex}{lex}

\DeclareMathOperator{\Exc}{Exc}
\DeclareMathOperator{\res}{res}

\DeclareMathOperator{\cone}{cone}
\DeclareMathOperator{\Supp}{Supp}
\DeclareMathOperator{\cent}{centre}
\DeclareMathOperator{\rrk}{rat.rank}

\def\rd#1.{\lfloor{#1}\rfloor}
\def\rp#1.{\lceil{#1}\rceil}

\newcommand{\N}{\mathbb N}
\newcommand{\Z}{\mathbb Z}
\newcommand{\Q}{\ensuremath{\mathbb Q}}
\newcommand{\R}{\mathbb R}

\renewcommand{\P}{\mathbb P}
\renewcommand{\O}{\mathscr O}

\renewcommand{\phi}{\varphi}
\renewcommand{\theta}{\vartheta}
\newcommand{\id}{\mathrm{id}}

\newcommand{\Bs}{\mathrm{Bs}}

\newcommand{\Rees}{\mathrm{Rees}}

\newcommand{\sO}{\mathscr{O}}

\DeclareMathOperator{\supp}{supp}
\newcommand{\M}{\mathrm M}
\newcommand{\GL}{\mathrm{GL}}
\newcommand{\SL}{\mathrm{SL}}
\DeclareMathOperator{\ddiv}{div}
\newcommand{\tY}{\widetilde{Y}}

\numberwithin{equation}{theorem}

\DeclareRobustCommand{\SkipTocEntry}[5]{}

\setitemize[1]{leftmargin=*,parsep=0em,itemsep=0.125em,topsep=0.125em}
\setenumerate[1]{leftmargin=*,parsep=0em,itemsep=0.125em,topsep=0.125em}

\newcommand{\iref}[3]{\the\value{#1}.\the\value{#2}(\the\value{#3})}

\renewcommand{\tilde}{\widetilde}

\newcommand{\sOX}{\ensuremath{\mathscr{O}^{}_{\! X}}}

\newcommand{\tildeX}{\ensuremath{\tilde{X}}}

\newcommand{\tX}{\tildeX}

\DeclareMathOperator{\Eff}{Eff}

\DeclareMathOperator{\Amp}{Amp}
\DeclareMathOperator{\Nef}{Nef}

\renewcommand{\P}{\ensuremath{\mathbb P}}
\newcommand{\NN}{\ensuremath{\mathbb N}}
\newcommand{\PP}{\ensuremath{\mathbb P}}
\newcommand{\QQ}{\ensuremath{\mathbb Q}}
\newcommand{\RR}{\ensuremath{\mathbb R}}
\newcommand{\ZZ}{\ensuremath{\mathbb Z}}
\newcommand{\OO}{\ensuremath{\mathcal O}}

\newcommand\lra{\longrightarrow}

\newcommand{\HH}[3]{\ensuremath{H^{#1}\left(#2,#3\right)}}

\newcommand{\vol}[2]{\ensuremath{{\rm vol}_{#1}( #2 ) } }

\newcommand{\st}[1]{\ensuremath{ \left\{ #1 \right\} }}

\newcommand{\deq}{\ensuremath{ \stackrel{\textrm{def}}{=}}}
\newcommand\ybul{Y_{\bullet}}

\def\.{\cdot}
\def\^{\widehat}
\def\~{\widetilde}

\newcommand\newop[2]{\def#1{\mathop{\rm #2}\nolimits}}
\newop\voll{vol}
\newop\NS{NS}
\newop\Neg{Neg}
\newop\Null{Null}
\newop\Pic{Pic}
\newop\Bstab{B_{+}}
\newop\Bst{B_{stab}}
\newop\Bres{B_{restr}}
\newop\Bplus{\mathbf{B}_+}
\newop\Bminus{\mathbf{B}_-}
\newop\Exc{Exc}
\newop\B{\mathbf{B}{}}
\newop\Bs{Bs}
\newop\End{End}
\newop\Amp{Amp}
\newop\Face{Face}
\newop\BigCone{Big}
\newop\index{ind}
\newop\reg{reg}

\newcommand{\equ}{\ensuremath{\,=\,}}

\newcommand{\dgeq}{\ensuremath{\,\geq\,}}

\newcommand{\dsubseteq}{\ensuremath{\,\subseteq\,}}

\newop\Gal{Gal}

\newcommand{\elx}{\ensuremath{\varepsilon(L;x)}}
\newcommand{\mlx}{\ensuremath{\mu(L;x)}}
\newcommand{\ilx}{\ensuremath{\iota(L;x)}}

\newcommand{\wht}[1]{\ensuremath{\widehat{#1}}}
\newcommand{\xh}{\wht{X}}
\newcommand{\lh}{\wht{L}}
\newcommand{\yh}{\wht{Y}}

\newcommand{\dht}{\wht{\Delta}}

\newcommand{\dyl}{\Delta_{\ybul}(L)}
\newcommand{\sigh}{\wht{s}}

\newcommand{\xqedhere}[2]{%
  \rlap{\hbox to#1{\hfil\llap{\ensuremath{#2}}}}}
\title{Concave transforms of filtrations and rationality of Seshadri constants}
\author{Alex K\"uronya}
\address{Alex K\"uronya, Institut f\"ur Mathematik, Goethe-Universit\"at Frankfurt, Robert-Mayer-Str. 6-10., D-60325 Frankfurt am Main, Germany}
\address{BME TTK Matematika Int\'ezet Algebra Tansz\'ek, Egry J\'ozsef u. 1., H-1111 Budapest, Hungary}
\email{\tt kuronya@math.uni-frankfurt.de}

\author{Catriona Maclean}
\address{Institut Fourier, Universit\'e Grenoble Alpes,  CS 40700,  38058 Grenoble cedex 9, France}
\email{\tt catriona.maclean@univ-grenoble-alpes.fr}           

\author{Joaquim Ro\'e}
\address{Departament de Matem\'atiques,	Facultat de Ci\'encies, C1/346, Universitat Auton\'oma de Barcelona 08193 Bellaterra (Barcelona) Spain}
\email{\tt jroe@mat.uab.cat}

\begin{document}

\begin{abstract}
	We show that the subgraph of the concave transform of a multiplicative  filtration on a section ring is the Newton--Okounkov body of a certain semigroup, and if the filtration is induced by a divisorial valuation, then the associated graded algebra is the algebra of sections of a concrete line bundle in higher dimension. We use this description to give a rationality criterion for certain Seshadri constants. Along the way we introduce Newton--Okounkov bodies of abstract graded semigroups and determine conditions for their slices to be Newton--Okounkov bodies of subsemigroups.
\end{abstract}

\maketitle

\section{Introduction}

\subsection{Background and motivation}	
This paper deals with Newton--Okounkov bodies and the associated  concave transforms of  
multiplicative filtrations.
 On the one hand  in Chapter 2 we generalise the theory of Newton--Okounkov bodies on subsemigroups of $\Z^n$ to the more general settings of subsemigroups of $\Z\times\R^n$ and abstract graded semigroups. 
 On the other, in Chapters 3 and 4 we use these generalizations to show that the subgraph of the concave transform on a Newton--Okounkov body of a filtration is itself the Newton--Okounkov body of an explicit semigroup and an explicit algebra, which in the case of a divisorial valuation is the Newton--Okounkov body of a concrete line bundle in higher dimension. Finally we use this description to give a rationality criterion for certain Seshadri constants. 
	
In the domain of algebraic geometry, Newton--Okounkov bodies are convex bodies associated to subalgebras of rational function fields of algebraic varieties. They  arose as a way of understanding asymptotic behaviour of lattice semigroups, and have become by now a near standard tool in the asymptotic theory of linear series on projective varieties with applications in arithmetic geometry, combinatorics, Diophantine approximation, mirror symmetry, and representation theory (for a sampling of applications the reader is invited to consult \cite{FFL,HK,RW} for instance). These  Newton--Okounkov bodies were defined by Kaveh--Khovanskii \cite{KK12} and Lazarsfeld--Musta\c{t}\u{a} \cite{LM09}, with 
both works building on earlier results of Okounkov \cite{Ok1,Ok2}. For the fundamentals of the theory, the reader can consult the original works, but also the expository papers \cite{Bou14,KL_Geom}.

\subsection{Main results}
Inspired by Boucksom's proposal to define Newton--Okounkov bodies with respect to valuations of maximal rational rank (which essentially consists in \emph{reembedding} the value group as a subgroup of $\Z^n$) and by Boucksom--Chen's construction of \emph{filtered} Newton--Okounkov bodies (which do not relate to a subsemigroup of a finitely generated group) in  Section 2 we will construct the Newton--Okounkov body of a graded cancellative torsion-free semigroup. In the case of a subsemigroup of $\Z^n$, it is well-known that the growth of its Hilbert function is governed by the volume of its Newton--Okounkov body; we generalize this result to the abstract setting as follows:

\begin{theoremA}[Theorem~\ref{thm:nob-semigroup-general}]
	Let $\Sigma$ be a graded, cancellative, torsion free semigroup, let $\Sigma_\Z$ bethe minimal abelian group containing $\Sigma$, and denote $H_\Sigma$ the Hilbert function of $\Sigma$.
	\begin{enumerate}
		\item If $\Sigma_\Z$ is finitely generated then $\rank\Sigma=n<\infty$, and
		\begin{itemize}
			\item if $\Sigma$ is linearly bounded, $H_\Sigma(d)=\vol{}{\Delta(\Sigma)}\,d^{n-1}+o(d^{n-1}),$
			\item otherwise, $\vol{}{\Delta(\Sigma)}=\lim(H_\Sigma(d)/d^{n-1})=\infty$.
		\end{itemize}
		\item If $\Sigma_\Z$ is not finitely generated and $\rank\Sigma=n<\infty$, then $\lim(H_\Sigma(d)/d^{n-1})=\infty$.
		\item If $\rank\Sigma=\infty$, then for every natural $n$, $\lim(H_\Sigma(d)/d^{n-1})=\infty$, i.e., the growth rate of $H_\Sigma$ is not polynomial.
	\end{enumerate} 
\end{theoremA}
\noindent
The application we have in mind for Newton--Okounkov bodies of semigroups $\Sigma$ such that $\Sigma_\Z$ is not finitely generated is to give a unified approach that encompasses filtered Newton--Okounkov bodies. This type of body was introduced by Boucksom and Chen in \cite{BC} with an ad-hoc construction that builds the convex body from slices, which are themselves bodies of semigroups which \emph{do embed} into finitely generated groups.
With our approach, filtered Newton--Okounkov bodies are regular Newton--Okounkov bodies (see Subsection \ref{sec:ct-via-graded-linear-series}), and we prove that their slices are, under a technical hypothesis (asymptotic convexity, \ref{def:asymptotically-convex}) Newton--Okounkov bodies of \emph{restricted subsemigroups}.
Such subsemigroups were already considered by 
Lazarsfeld--Musta\c{t}\u{a} in the particular case when $\Sigma$ is the value semigroup of a graded linear series under a valuation of maximal rank.
In that case, the restricted semigroups have a geometric meaning, corresponding to restricted linear series.
As a consequence of our work we obtain an integral formula (Corollary \ref{cor:volume-and-slices}, in terms of restricted subsemigroups) for the volume of every asymptotically convex semigroup $\Sigma$, widely generalizing the one for filtered Newton--Okounkov bodies given in \cite{BC}.
The relationship between the global Newton--Okounkov body of a variety \cite[Theorem 4.5]{LM09} and the bodies of individual divisor classes is also a case of slicing with respect to restricted subsemigroups (Remark \ref{rem:global}).
We refer for notation and further details to Subsections 2.A and 2.B.

Let us summarise what is known about the convex geometry of Newton--Okounkov bodies arising in algebraic geometry. 
We know that the Newton--Okounkov bodies of full linear series are always polygons in dimension two \cite{KLM,AKL}, and  that they are not polyhedral in higher dimensions in general \cite{KLM}  unless some  strong finite generation condition is present \cite{AKL} (see also  \cite{PU}). On the other hand, any convex set can appear as the Newton--Okounkov body of a graded linear series \cite{LM09}. It has been conjectured that every line bundle possesses a Newton--Okounkov body which is a semi-algebraic set \cite{KLM_volume,KL_Geom}. 

Let $X$ be a projective variety of dimension $n$ over an algebraically closed field $K$.  To a line bundle  $L$ on $X$ and  a 
full flag of subvarieties $\ybul$ subject to some mild nondegeneracy conditions one can associate the appropriate Newton--Okounkov body $\dyl\subseteq \RR^n$. If $L$ is big  then the resulting convex body  (as a subset of $\RR^n$) only depends on the numerical equivalence class of $L$. Conversely, the association $\ybul \to \dyl$ yields a universal numerical invariant for big line bundles \cite{Jow}.  Following this train of thought the authors of \cite{KL_noninf,KL_inf,KL_Reider} and \cite{Roe} studied the local positivity of line bundles. 

Although most of the development for projective varieties focuses on big  divisors, the papers \cite{CHPW,CHPW2} extended many of the results to the 
pseudo-effective case. It should be noted that in these papers, as in the definition of the global Newton--Okounkov bodies in 
\cite{LM09}, the definition  of the Newton--Okounkov body of a non-big divisor  necessarily differs from that given in the big case, which does not generally
give a numerical invariant in the non-big case. The Newton Okounkov body of a non-big pseudo-effective divisor $D$ is therefore defined as the limit of the 
Newton--Okounkov bodies of $D+\epsilon A$ for positive $\epsilon$ and ample $A$. Note that the construction of Newton--Okounkov bodies as a function of numerical equivalence classes is not continuous in general as one approaches the boundary of the pseudo-effective cone.   
For any Mori dream space $X$, Postinghel and Urbinati in  \cite{PU} find a flag on $X$  with respect to which the global Newton--Okounkov body $\Delta_{\ybul}(X)$ over 
$\overline{\Eff}(X)$ is rational polyhedral. 		
 
The Newton--Okounkov body of a line bundle can be seen as a generalisation to arbitrary varieties of the toric polytope of a line bundle on a toric variety.
An analogue of the moment map on these polytopes --- the concave transform of a multiplicative filtration on the section ring --- was introduced by Boucksom--Chen \cite{BC} and independently by  Witt-Nystr\"om \cite{WN} and further studied in \cite{BKMS} (see also \cite{KMSwB}). 

Multiplicative filtrations on sections rings arise naturally in various ways. One immediate example is to consider the order of vanishing along a smooth subvariety. 
One of our main results, Theorem C,  links knowledge about the order of vanishing filtration to the rationality of Seshadri constants and hence to the 
conjectures of Nagata and Segre--Harbourne--Gimigliano--Hirschowitz (cf. \cite{DKMS}).
 
Donaldson's test configurations \cite{Don} are another source of  multiplicative filtrations \cite{RT1,RT2,Szek,WN}. 
Donaldson \cite{Don} studies the link between K-stability and constant scalar curvature metrics on toric surfaces and proves a weaker version of the
Donaldson-Tian conjecture: a key ingredient of this work is the use of a toric polytope of a line bundle of a toric threefold whose rational points encode, amongst other things, the Futaki invariant\footnote{Donaldson  in fact constructs this toric polytope in all dimensions.}
In \cite{WN} this  polytope is re-interpreted as the graph of a concave transform on a  Newton--Okounkov body of a multiplicative filtration arising from the test 
configuration, which enables Witt-Nystr\"om to generalise Donaldson's toric polytope  construction to arbitrary varieties.

In Chapter 4 we show that in the case of a divisorial multiplicative filtration the subgraph of the concave transform is again a Newton-Okoukov body in higher dimension. 

\begin{theoremB}[Theorem~\ref{thm:integrals are volumes}]
Let $X$ be a projective variety, $L$ a big line bundle on $X$,  $v$ a valuation  of maximal rational rank $n=\dim X$, $w$ a divisorial valuation on $K(X)$. Then there exists a projective variety $\xh$ of dimension $n+1$, a valuation of maximal rational rank $\widehat{v}$ on $\xh$, and a big  line bundle $\lh$ on $\xh$ such that the subgraph of the function $w\colon \Delta_{v}(L)\to \RR_{\geq}$  arising from $w$ equals the Newton--Okounkov body $\Delta_{\widehat{v}}(\lh)$. 	
\end{theoremB}

As an application of this result we give a sufficient condition for the 
rationality of Seshadri constants on surfaces. The question of the rationality 
of Seshadri constant has been present ever since  they were first defined by Demailly in \cite{Dem92}. The paper \cite{DKMS} 
added to the significance of the issue by proving that rationality of Seshadri 
constants on certain surfaces would disprove that  Segre--Harbourne--Gimigliano--Hirschowitz conjecture. 

Even though several asymptotic invariants of line bundles turned out to be 
rational in dimension two, this is far from clear for Seshadri constants. We 
use the above theorem to link rationality of volumes on threefolds to 
rationality of
Seshadri constants. This is a territory where not much is known: some volumes 
on threefolds are irrational \cite{Cut,Laz04I} (cf. \cite{KLM_volume} as well).

\begin{theoremC}[Corollaries~\ref{cor:fg implies Sehsadri rtl} and \ref{cor:rationality of Seshadri}]
Let $X$ be a smooth projective surface, $x\in X$, and let $L$ be an ample line bundle on $X$. Let $\tilde{X}$ be the blow-up of $X$ at $x$. 
\begin{enumerate}
	\item There exists a $\PP^1$-bundle $\xh$ over $\tilde{X}$ and a big line bundle $\lh$ on $\xh$ such that $\elx$ is rational provided $\vol{\xh}{\lh}$ is. In particular, this holds if 
$R(\xh,\lh)$ is finitely generated.
	\item 
	If there exists a positive integer $b$ satsifying $\mlx < b < \epsilon(L-K_X;x)-2$, then 
	$\elx\in \QQ$. 
\end{enumerate}
\end{theoremC}

\subsection{Organization of the article}
The article has arguably a somewhat expository flavour at places. Part of the time we treat material that is not far from the existing literature, nevertheless, we believe that our more general framework and slightly different point of view justifies our approach. 

This being said, Section 2 is devoted to the more abstract part of the paper dealing with the construction of Newton--Okounkov bodies of abstract semigroups and adjusting the results of Kaveh--Khovanskii to our setting. In the later subsections we discuss the case of ordered semigroups, and with it, the role of valuations, which leads to a detailed discussion of the construction of Newton--Okounkov bodies of line bundles on projective varieties. In Section 3 we describe concave transforms and their relationship with Rees algebras of filtrations. Finally, Section 4 hosts the explicit demonstration that subgraphs of concave transforms of multiplicative filtrations are in fact Newton--Okounkov bodies of line bundles in dimension one higher, and an application of this fact to the rationality of Seshadri constants. 
 
\subsection{Notation and conventions.}
All groups and semigroups in this paper are commutative and written in additive notation. All rings are commutative with identity. When working with varieties, we will be doing so over an arbitrary algebraically closed field except in Subsection 4.B. Large parts of the algebro-geometric material in the paper work for varieties over an arbitrary field, but we do not pursue minimal hypotheses in this direction.

\subsection*{Acknowledgements} We are grateful to  Christian Haase, Vlad Lazi\'c, Victor Lozovanu,  Matthias Nickel, Mike Roth  and  Lena Walter for helpful discussions. The second author was partially supported by ERC grant ALKAGE. The first and third authors gratefully acknowledge partial  support from the LOEWE Research Unit 'Uniformized Structures in Arithmetic and Geometry', and the Mineco Grant No. MTM2016-75980-P,  while the first author also enjoyed partial support from the NKFI Grant No. 115288 'Algebra and Algorithms', and the third also from AGAUR 2017SGR585. Our  project was initiated during the workshop 'Newton--Okounkov Bodies, Test Configurations, and Diophantine Geometry' at the Banff International Research Station. We  appreciate the stimulating atmosphere and the excellent working conditions at BIRS.

\section{Convex objects associated to semigroups and filtrations}
\label{sec:convex-objects-semigroups} 
Our purpose here  is to define Newton--Okounkov bodies for graded cancellative torsion-free semigroups. 
These are the abstract semigroups which can be embedded in $\R^n$; we extend the construction of Kaveh and Khovanskii \cite{KK12}, which works for semigroups $\Sigma$ embedded in $\Z^n\subset\R^n$ such that the generated abelian group $\Sigma_\Z$ equals $\Z^n$,  by 
\begin{enumerate*}[label=\itshape\alph*\upshape)]
	\item allowing arbitrary groups $\Sigma_\Z$, and
	\item showing independence from reembeddings, as long as these are \emph{full} (a technical condition essentially meaning that the $\R$-linear span of $\Sigma$ is of maximal dimension).
\end{enumerate*}
In Theorem \ref{thm:nob-semigroup-general} we establish that the Newton--Okounkov bodies obtained this way satisfy some of the most important properties proven by Kaveh--Khovanskii in \cite{KK12}; indeed, if $\Sigma_\Z$ is finitely generated then the volume of the body governs the growth rate of the semigroup. 
When $\Sigma_\Z$ is not finitely generated, the body and its volume are still important invariants of the semigroup, and in fact they are our main tool to approach concave transforms of filtrations in section \ref{sec:concave-transforms}.

The Newton--Okounkov bodies of restricted linear series introduced by Lazarsfeld--Musta\c{t}\u{a} in \cite{LM09} turn out to have an underlying semigroup-theoretic base; 
in subsection \ref{sec:abstract-volume} we introduce restricted semigroups of ordered semigroups, and give an integral formula for the volume in terms of restricted volumes (Corollary \ref{cor:volume-and-slices}). 

We study in greater detail Newton--Okounkov bodies of \emph{ordered} semigroups, such as those obtained from valuations and filtrations, which are the most relevant in algebraic geometry. 
In that case the some of the restricted semigroups, and accordingly some \emph{slices} of the Newton--Okounkov bodies are also invariants of the semigroup.
The section concludes by introducing, in our setting, the Newton--Okounkov bodies of line bundles in projective varieties.
\subsection{Newton--Okounkov sets and slices}\label{sec:sub-Rn}

We start by recalling the definition of Newton--Okounkov bodies of semigroups as introduced by Kaveh-Khovanskii.
To the best of our knowledge, previous work on Newton--Okounkov bodies of semigroups requires that these are subsemigroups of some lattice in a finite-dimensional vector space.
For the purposes of this paper we need to relax this hypothesis and allow arbitrary subsemigroups of finite-dimensional vector spaces. 
Therefore, even though our presentation follows the spirit of Boucksom \cite{Bou14} and Kaveh-Khovanskii \cite{KK12}, we will allow this added generality from the beginning.

We introduce the notions of \emph{asymptotically convex semigroup} and \emph{restricted semigroups}, which will play a key role in our interpretation of the \emph{slices} of Newton--Okounkov bodies appearing in \cite{LM09}, \cite{Bou14}, \cite{BC}.
The main result in this subsection is Theorem \ref{thm:restricted_body_semigrp}, in which we prove that slices of the Newton--Okounkov body of an asymptotically convex semigroup are Newton--Okounkov bodies of restricted semigroup.

\begin{dfn} A \emph{grading} on a group $\Gamma$ (resp. a semigroup $\Sigma$) is a homomorphism $\deg:\Gamma\rightarrow \Z$ (resp. $\Sigma\rightarrow \Z$). 
	We assume throughout that gradings are surjective.
	An \emph{embedding} of semigroups is an injective homomorphism.
	A \emph{graded group} (resp. \emph{semigroup})
	is a group (resp. a semigroup) with a fixed grading.

\end{dfn}
\begin{ntn}
	Let $V$ be a finite-dimensional real vector space, and $X\subset V$ a subset.
	The subsemigroup (respectively, subgroup, linear span and $\Q$-linear span) of $V$ generated by $X$ will be denoted $\langle X\rangle_\N$ (respectively, $\langle X\rangle_\Z$, $\langle X\rangle_\R$, $\langle X\rangle_\Q$).
	The convex cone generated by $X$ will be denoted by 
	$\cone(X)=\{a_1x_1+\dots+a_nx_n|a_i\ge0, x_i\in X \}$, and the topological closure by $\overline{X}$.
\end{ntn}

\begin{dfn}	
	Let $V$ be a finite-dimensional real vector space, and $\Sigma$ a graded subsemigroup of $V$.
	$\Sigma$ is said to be \emph{linearly graded} if the grading map $\deg:\Sigma\rightarrow\Z$
	extends to a surjective linear form $\deg:V\rightarrow\R$.
	In this case we use the notation $L_d$ for the hyperplane
	$\{x\in V | \deg(x)=d\}$.
	If $V=\R^r$ and $\deg(x_1,\dots,x_r)=x_1$ then we say that $\Sigma$ is graded by first component.
		
	A linearly graded subsemigroup $\Sigma$ of $V$ is said to be 
	\emph{linearly bounded} if there is a basis $v_1,\dots,v_n$ of $V$ formed by vectors of positive degree such that $\Sigma$ is contained in the positive orthant $\cone(v_1,\dots,v_r)$.
\end{dfn}

\begin{dfn}
	The \emph{Newton--Okounkov set} $\Delta(\Sigma)$ 
	of a linearly graded subsemigroup $\Sigma\subset V$ is the topological closure
	\[ \Delta(\Sigma)=\overline{\left\{\left.
		\frac{\sigma}{\deg(\sigma)}
		\,\right|\, \sigma \in \Sigma \setminus \deg^{-1}(\{0\})\right\}} \subset 
	L_1.\]
	Note that if $\Sigma$ is linearly bounded then 0 is the only element of degree 0 in $\Sigma$. 
	It is not hard to see that $\Delta(\Sigma)$ is convex, and if $\Sigma$ is linearly bounded then $\Delta(\Sigma)$ is compact; more precisely, there is an equality
	\[ \Delta(\Sigma)= \overline{\cone(\Sigma)}\cap L_1.\]
	Moreover, $\Delta(\Sigma)$ has nonempty interior if and only if the linear span $\langle\Sigma\rangle_\R$ is $V$. 
	
	If $\Sigma$ is linearly bounded and $\langle\Sigma\rangle_\R=V$, then $\Delta(\Sigma)$ is called the Newton--Okounkov \emph{body} of $\Sigma$.
\end{dfn}

The following result, originating in Khovanskii's work \cite{Kho95}, underlies much of the theory of Newton--Okounkov bodies. 
It tells us that, if the group $\langle\Sigma\rangle_\Z$ is a lattice in $V$, then $\cone(\Sigma)\cap \langle\Sigma\rangle_\Z$ is asymptotically (i.e., for large degrees) a good approximation of $\Sigma$.

\begin{theorem}[{\cite[Théorème 1.3]{Bou14}}, {\cite[Theorem 1.6]{KK12}}]\label{thm:approx_semigroup}
	Let $\Sigma\subset V$ be a linearly graded subsemigroup such that $\langle\Sigma\rangle_\Z$ is a lattice in $\langle\Sigma\rangle_\R$. 
	Let $C\subset \overline{\cone(\Sigma)}$ be a closed strongly convex cone that intersects the boundary of $\overline{\cone(\Sigma)}$ only at the origin.
	Then there is a constant $N>0$ such that each $\gamma\in C\cap \langle\Sigma\rangle_\Z$ with $\deg(\gamma)\ge N$ belongs to $\Sigma$.
\end{theorem}

In particular, if $\sigma\in\langle\Sigma\rangle_\Z$ belongs to the interior of  $\overline{\cone(\Sigma)}$, it follows by  applying the theorem to the ray $\cone(\sigma)$ that there is a multiple $k\sigma$ belonging to $\Sigma$.
Semigroups not included in a lattice often do not have this property, but if they do, then their Newton--Okounkov sets behave not unlike usual Newton--Okounkov bodies of lattice subsemigroups. 
Thus we make the following definition.

\begin{definition}\label{def:asymptotically-convex}
	A subsemigroup $\Sigma\subset V$ will be called \emph{asymptotically convex} if, for every $\sigma\in\langle\Sigma\rangle_\Z$ belonging to the interior of  $\overline{\cone(\Sigma)}$, there is a multiple $k\sigma$ belonging to $\Sigma$.
\end{definition}

\begin{ntn}
	Let $\Sigma\subset V$ be a linearly graded subsemigroup, and  $W\subset L_0$ a linear subspace. We denote $\Sigma/W$ the image of $\Sigma$ in $V/W$ by the natural projection $p_W:V\rightarrow V/W$. Because elements of $W$ have degree $0$, the grading descends, and $\Sigma/W$ is a linearly graded subsemigroup. 
	For every  $\sigma\in\Sigma$, we denote
	\[\Sigma|_{W+\sigma}=\Sigma\cap\langle W+\sigma\rangle_\R\]
	the \emph{restricted semigroup} determined by $W$ and $\sigma$. 
	
	We call a linear $\R$-subspace $W\subset V$ \emph{$\Sigma$-rational} if it can be generated by vectors in  $\langle\Sigma\rangle_\Z$.
\end{ntn}

If $\Sigma$ is asymptotically convex, the \emph{slices} of its Newton--Okounkov body in the direction of a $\Sigma$-rational subspace $W$ are Newton--Okounkov bodies, namely those of the restricted semigroups $\Sigma|_{W+\sigma}$:

\begin{theorem}\label{thm:restricted_body_semigrp}
	Let $\Sigma\subset V$ be an asymptotically convex, linearly graded, linearly bounded, subsemigroup, and $W\subset L_0$ a $\Sigma$-rational linear subspace.
	Denote $\overline{V}=V/W$, let $p:V\rightarrow \overline{V}$ be the projection, and for each $\mathbf{v}\in \overline{V}$, denote $L_\mathbf{v}$ the affine space $p^{-1}(\mathbf{v})\subset V$.
		
	\begin{enumerate}
		\item The image of $\Delta(\Sigma)\subset L_1$ by the projection $p$ is $\Delta(\Sigma/W)\subset \overline{L}_1$, where $\overline{L}_1$ denotes the hyperplane of vectors of degree 1 in $\overline{V}$.
		\item For every vector $\mathbf{v}\in\langle\Sigma\rangle_\Q$ in the relative interior of $\Delta(\Sigma/W)$, the slice $\Delta(\Sigma)\cap L_{\mathbf{v}}$  is the Newton--Okounkov set of the restricted semigroup $\Sigma|_{W+\sigma}$ for a suitable $\sigma \in \Sigma$.
	\end{enumerate}
	
\end{theorem}
\begin{proof}
	If $\Sigma\subset \langle W+\sigma\rangle_\R$ for some $\sigma\in\Sigma$ (and hence for all), then $\langle p(\Sigma)\rangle_\R$ is one-dimensional, $\Delta(\Sigma/W)$ is a single point, and the claims are obvious. 
	So we assume that $\Sigma\not\subset \langle W+\sigma\rangle_\R$ for all $\sigma\in\Sigma$
	
	The first claim is immediate from the definitions, observing that $p(L_1)=\overline{L}_1$.
	For the second claim, note that since $W$ is $\Sigma$-rational and $\Sigma$ is asymptotically convex, $\mathbf{v}$ belongs to the image of the map
	\begin{align*}
	\pi:\Sigma &\longrightarrow \Delta(\Sigma/W)\\
	\sigma &\longmapsto p(\sigma)/\deg(\sigma)
	\end{align*}
	Moreover, the fiber of $\pi$ 
	over $\mathbf{v}=\pi(\sigma) \in \pi(\Sigma)$ is exactly the restricted semigroup $\Sigma|_{W+\sigma}$.
	We will prove that, given an element $\sigma\in \Sigma$ such that $\mathbf{v}=\pi(\sigma)$ belongs to the interior of $\Delta(\Sigma)$, the equality 
	\[\Delta(\Sigma)\cap p^{-1}(\mathbf{v})=\Delta(\Sigma|_{W+\sigma})\]
	holds.
	 
	Consider the linear space $H=\langle W+\sigma\rangle_\R=p^{-1}(\langle \mathbf{v}\rangle_\R)$. As $W$ is $\Sigma$-rational, we have $H=\langle\Sigma|_{W+\sigma}\rangle_\R$.
	Let $C_{W,\sigma}=\overline{\cone(\Sigma|_{W+\sigma})}\subset H\subset V$, and $C=\overline{\cone(\Sigma)}$, so that $\Delta(\Sigma|_{W+\sigma})=C_{W,\sigma} \cap L_1$ and
	$\Delta(\Sigma)=C \cap L_1$.
	Obviously $C_{W,\sigma}\subseteq C \cap H$, therefore the claim will follow by proving the reverse inclusion $C\cap H \subset C_{W,\sigma}$.
	Since $W$ is $\Sigma$-rational, $C\cap H \subset\overline{C \cap H\cap \langle\Sigma\rangle_\Z}$, and since $\Sigma$ is asymptotically convex, $C \cap H\cap \langle\Sigma\rangle_\Z\subset \overline{\cone(\Sigma\cap H)}$. Now the claim follows.
\end{proof}

\begin{remark}
	Note that the restricted semigroup $\Sigma|_{W+\sigma}$ is linearly bounded, and it follows from the theorem that if $p(\sigma)/\deg(\sigma)$ belongs to the relative interior of $\Delta(\Sigma/W)$ then the linear span of $\Sigma|_{W+\sigma}$ is $\langle W+\sigma\rangle_\R$.
	So the Newton--Okounkov set of $\Sigma|_{W+\sigma}$ is a body in $\langle W+\sigma\rangle_\R$, and we call it the \emph{restricted Newton--Okounkov body} of $\Sigma|_{W+\sigma}$. 
\end{remark}

\begin{remark}\label{rmk:brunn-minkowski} 
	Now assume that the vector space $V$ is endowed with a volume form (for instance, $V=\R^n$ with the Euclidean form). 
Then $\vol{}{\Delta(\Sigma)}$ is an invariant of $\Sigma$, which is especially important when $\langle\Sigma\rangle_\Z$ is a finitely generated group, as will be shown in the next subsection.
Theorem \ref{thm:restricted_body_semigrp} has strong consequences on how the volumes $\vol{}{\Delta(\Sigma|_{W+\sigma})}$ depend on $\sigma\in \Sigma$.
Namely, the map $\sigma\mapsto \vol{}{\Delta(\Sigma|_{W+\sigma})}^{1/\dim W}$ factors  as the composition of two maps
\begin{align*}
\Sigma\setminus\{0\} & \rightarrow \Delta(\Sigma/W)& \Delta(\Sigma/W) &\rightarrow \R\\
\sigma& \mapsto \frac{p(\sigma)}{\deg(\sigma)}&
\mathbf{v}& \mapsto \varphi(\mathbf{v})
\end{align*}
with $\phi$ upper semicontinuous, concave (by the Brunn-Minkowski inequality) and continuous in the interior of $\Delta(\Sigma/W)$.
\end{remark}

\subsection{Newton--Okounkov bodies of abstract semigroups. Volume and Hilbert function}
\label{sec:abstract-volume}
One of the principal results in the theory of Newton--Okounkov bodies is Theorem \ref{thm:volNOB_semigroup_KK} below, which relates the volume of the body $\Delta(\Sigma)$ with the growth rate of the Hilbert function of $\Sigma\subset V$ when the generated subgroup is a lattice.
However, the Hilbert function only depends on the abstract graded semigroup $\Sigma\rightarrow\Z$, not on its being embedded in a vector space. 
In this section we study the dependence of the Newton--Okounkov bodies of a given graded semigroup $\Sigma$ on the choice of embeddings in real vector spaces.
We show that the Newton--Okounkov body determined by a \emph{full embedding} (Definition \ref{def:full-embedding}) is essentially independent of the particular choice of embedding, which will allow us to define the Newton--Okounkov body of an abstract graded cancellative and torsion-free semigroup (see Lemma \ref{lem:sl_invariance} below).

We show that such bodies determine the rate of growth of $H_\Sigma$ if and only if $\Sigma_\Z$ is finitely generated (Theorem \ref{thm:nob-semigroup-general}), and we also investigate the interplay of the volume of the Newton--Okounkov body with the rate of growth of restricted subsemigroups, obtaining an integral formula for the volume in the case of an embedded semigropup $\Sigma\subset V$, which does not need $\Sigma_\Z$ to be finitely generated and generalizes the Boucksom--Chen formula \cite[Corollary 1.13]{BC} for filtered Newton--Okounkov bodies.

\begin{definition}
	Let $\Sigma$ be a graded semigroup.
	For each $d\ge 0$, denote $\Sigma_d=\{\sigma\in\Sigma|\deg(\sigma)=d\}$ and $H_\Sigma(d)=|\Sigma_d|\in\N\cup\{\infty\}$. 
	$H_\Sigma$ is called the \emph{Hilbert function} of $\Sigma$.
	If $\Sigma$ is a linearly graded, linearly bounded, subsemigroup of a real vector space $V$ of finite dimension, and $\langle\Sigma\rangle_\Z$ is finitely generated, then $H_\Sigma(d)$ is finite for all $d$.
\end{definition} 
\begin{theorem}[Kaveh-Khovanskii, 
	see {\cite[Corollary 1.16 and Theorem 1.18]{KK12}}
	or {\cite[Théorème 1.12 and Corollaire 1.14]{Bou14}}]\label{thm:volNOB_semigroup_KK}
	Let $V$ be a real $r$-dimensional vector space endowed with a volume form.
	Let $\Sigma\subset V$ be a linearly graded subsemigroup such that $\langle\Sigma\rangle_\Z$ is a lattice. 
	Denote $\det^1(\Sigma)$ the determinant of the lattice $\langle\Sigma\rangle_Z\cap L_1$ with respect to the volume form induced on $L_1$.
	If $\Sigma$ is linearly bounded then
	\[H_\Sigma(d)=\frac{\vol{}{\Delta(\Sigma)}}{\det^1(\Sigma)}d^{r-1}+o(d^{r-1}).\]
	If $\Sigma$ is not linearly bounded then $\Delta(\Sigma)$
	has infinite volume and $\lim H_\Sigma(d)/d^{r-1}=\infty$.
\end{theorem}

\noindent
Note that $\det^1(\Sigma)$ is simply the volume of the smallest parallelepiped in $L_1$ with vertices on $\langle\Sigma\rangle_\Z$.

We now fix a graded semigroup $\Sigma$, and study the Newton--Okounkov sets obtained from it by different embeddings.
The next lemma collects some elementary facts on embeddings which will be useful to describe the effect of reembedding on the Newton--Okounkov bodies. 	

\begin{lemma}\label{lem:sgrp-embeddings}
	\begin{enumerate}
		\item  	A semigroup $\Sigma$ can be embedded in a group if and only if it is cancellative ($\sigma_0+\sigma_2=\sigma_1+\sigma_2$ implies $\sigma_0=\sigma_1$).  $\Sigma$ can be embedded in a rational or real vector space if and only if it is cancellative and torsion free.
	\end{enumerate}
	\noindent\emph{In this case we denote $\Sigma_\Z$ the minimal abelian group containing $\Sigma$ (unique up to isomorphism).	The rational rank\footnote{To avoid confusion, later in the paper, with the rank of a valuation, which is the order rank of its value group, the rank of a group will always we called \emph{rational rank}.} of a cancellative semigroup $\Sigma$ is defined as $\rrk\Sigma\deq\rrk \Sigma_\Z\deq\dim_\Q (\Sigma_\Z \otimes_\Z \Q)$.}
	\begin{enumerate}[resume]
		\item 	Let $\Sigma$ be a cancellative semigroup of finite rational rank. There exists an embedding $\iota:\Sigma\hookrightarrow\R^r$ with $\langle\iota(\Sigma)\rangle_\R=\R^r$ if and only if $\Sigma$ is torsion free and $r \le \rrk \Sigma$.
		\item	If $\Sigma$ is a graded, cancellative, torsion free semigroup with finite rational rank, then for every embedding $\iota:\Sigma\hookrightarrow\R^{\rrk\Sigma}$ such that $\langle\iota(\Sigma)\rangle_\R=\R^{\rrk\Sigma}$,
		the grading on $\Sigma$ induces a linear grading on $\iota(\Sigma)$.
		\item	If $\Sigma$ is a graded, cancellative, torsion free semigroup with finite rational rank, and $\iota_1, \iota_2:\Sigma\rightarrow\R^{\rrk\Sigma}$ are two embeddings such that  $\langle\iota_i(\Sigma)\rangle_\R=\R^{\rrk\Sigma}$,
		then $\iota_1(\Sigma)$ is linearly bounded if and only if $\iota_2(\Sigma)$ is linearly bounded.
	\end{enumerate}
\end{lemma}

\begin{definition}\label{def:full-embedding}
	Motivated by the last two properties in the lemma above, we call an embedding $\iota:\Sigma\hookrightarrow V$, where $V$ is a finite-dimensional real vector space, \emph{full} if $\dim_\R (V)=\rrk\Sigma$ and $\langle\iota(\Sigma)\rangle_\R=V$. 
	We shall say that a graded, cancellative, torsion free semigroup with finite rational rank is linearly bounded if its image by a full embedding is linearly bounded.
\end{definition}

\begin{remark}
	If $\iota:\Sigma\hookrightarrow V$ is a full embedding of a graded, cancellative, torsion free semigroup with finite rational rank, since the induced grading on $\iota(\Sigma)$ is linear, we may compose it with an appropriate isomorphism $V\rightarrow\R^{\rrk\Sigma}$ so that $\iota(\Sigma)$ is graded by first component in $\R^{\rrk\Sigma}$.
	Further, if $\Sigma_\Z$ is finitely generated, then a full embedding induces an embedding as a lattice $\Sigma_\Z\hookrightarrow V$.
	Again composing with a suitable isomorphism $V\rightarrow\R^{\rrk\Sigma}$ we may assume that $\Sigma_\Z\cong\langle\iota(\Sigma)\rangle_\Z=\Z^{\rrk\Sigma}\subset \R^{\rrk\Sigma}$, so $\iota(\Sigma)$ is one of the semigroups originally considered by Kaveh--Khovanskii.
\end{remark}

\begin{lemma}\label{lem:sl_invariance}
	Let $\Sigma$ be a graded, cancellative, torsion free semigroup with $\rrk\Sigma=r<\infty$. Let $\iota_1,\iota_2:\Sigma\rightarrow \R^r$ be two embeddings,
	and denote $\Delta_1(\Sigma)$, $\Delta_2(\Sigma)$
	the respective Newton--Okounkov bodies. Assume that
	\begin{enumerate}
		\item $(\iota_j)_1(\sigma)=\deg(\sigma)$ for every 
		$\sigma\in \Sigma$. In other words, $\Sigma$ is
		graded by first component for both embeddings. 
		\item $\langle\iota_j(\Sigma)\rangle_\Z=\Z^{r}\subset \R^{r}$ for both $j=1, 2$.
	\end{enumerate}
	Then there is an automorphism with integer coefficients $\phi\in \GL_{r-1}(\Z)\subset\GL_{r-1}(\R)$ of $\R^{r-1}$
	such that $\phi(\Delta_1(\Sigma))=\Delta_2(\Sigma)$.
\end{lemma}
\begin{proof}
	Slightly abusing notation, we denote the extension of $\iota_j$ to $\Sigma_\Z\rightarrow\Z^{\rrk\Sigma}$, which is an isomorphism for both $j=1, 2$ with the same symbol $\iota_j$.
	Consider the automorphism $\varphi_+=\iota_2\circ\iota_1^{-1}:\Z^r\rightarrow\Z^r$.
	By Lemma \ref{lem:sgrp-embeddings} and the hypotheses, both embeddings are linearly graded by first component, and $\iota_i(\Sigma_\Z)\cap L_1=\{1\}\times\Z^{r-1}$.  
	Hence, $\varphi_+$ restricts to an automorphism $\phi$ of $\{1\}\times\Z^{r-1}$, and it is immediate from the definiton of the Newton--Okounkov set that $\phi(\Delta_1(\Sigma))=\Delta_2(\Sigma)$.
\end{proof}
This allows the following definition.

\begin{definition}[Newton--Okounkov body of an abstract semigroup]\label{def:nob-abstract-semigroup}
	Given a graded cancellative commutative semigroup $\Sigma$ such that 
	$\Sigma_\Z$ is finitely generated and torsion free, we define
	the Newton--Okounkov body $\Delta(\Sigma)$ as the Newton--Okounkov body
	of any embedding into $\R^r$ satisfying the properties of lemma \ref{lem:sl_invariance}, modulo
	the action of $\GL_{r-1}(\Z)$.

	Elements in $ \GL_{r-1}(\Z)$ have determinant $\pm1$. Therefore, for every graded, cancellative, torsion free semigroup of finite rational rank, $\vol{}{\Delta(\Sigma)}$ is a well defined positive real number.
\end{definition}

We can now prove our main Theorem on the growth of Hilbert functions of abstract graded semigroups.

\begin{theorem}\label{thm:nob-semigroup-general}
	Let $\Sigma$ be a graded, cancellative, torsion free semigroup.
	\begin{enumerate}
		\item If $\Sigma_\Z$ is finitely generated then $\rrk\Sigma=r<\infty$, and
		\begin{itemize}
			\item if $\Sigma$ is linearly bounded, $H_\Sigma(d)=\vol{}{\Delta(\Sigma)}\,d^{r-1}+o(d^{r-1}),$
			\item otherwise, $\vol{}{\Delta(\Sigma)}=\lim(H_\Sigma(d)/d^{r-1})=\infty$.
		\end{itemize}
		\item If $\Sigma_\Z$ is not finitely generated and $\rrk\Sigma=r<\infty$, then $\lim(H_\Sigma(d)/d^{r-1})=\infty$.
		\item If $\rrk\Sigma=\infty$, then for every natural number $r$, $\lim(H_\Sigma(d)/d^{r-1})=\infty$, i.e., the growth rate of $H_\Sigma$ is not polynomial.
	\end{enumerate} 
\end{theorem}

\begin{proof}
	The first statement follows immediately from Theorem \ref{thm:volNOB_semigroup_KK} applied to any full embedding $\iota$ that maps $\Sigma_\Z$ to $\Z^r$
	
	For the second statement, if $\Sigma_\Z$ is not finitely generated then we can find finitely generated subsemigroups of the same rational rank $r$, $\Sigma_1\subset\dots\subset\Sigma_i\subset\dots\subset\Sigma$ with each $(\Sigma_i)_\Z$ strictly contained in $(\Sigma_{i+1})_\Z$ and such that the restricted degree map $\deg|_{(\Sigma_{i+1})_\Z}$ is surjective for all $i$.
	Fix a full embedding $\iota:\Sigma\hookrightarrow\R^r$.
	Because $\rrk \Sigma_i=\rrk \Sigma=r$, it follows that $(\Sigma_i)_\Z\otimes_\Z \Q=\Sigma_\Z\otimes_\Z \Q\cong\Q^r$ and then $\langle\iota(\Sigma)\rangle_\R=\R^{r}$ guarantees that $\langle\iota(\Sigma_i)\rangle_\R=\R^{r}$, in particular the number $v_1=\vol{}{\Delta(\Sigma_1)}$ is positive.
	Moreover the semigroup inclusions give $\vol{}{\Delta(\Sigma_i)}\ge v_1$ for all $k$.
	
	Since $(\Sigma_i)_\Z$ has index at least 2 in $(\Sigma_{i+1})_\Z$, and the grading is fixed, $\langle\iota(\Sigma_i)\rangle_\Z\cap L_1$ has index at least 2 in $\langle\iota(\Sigma_{i+1})\rangle_\Z\cap L_1$.
	Hence $\det^1(\iota(\Sigma_{i+1}))\le \det^1(\iota(\Sigma_i))/2\le \det^1(\Sigma_1)/2^i$ and hence by Theorem \ref{thm:volNOB_semigroup_KK},
	\[H_\Sigma(d)\ge H_{\Sigma_i}(d)=\frac{\vol{}{\Delta(\Sigma_i)}}{\det^1(\iota(\Sigma_i))}\,d^{r-1}+o(d^{r-1})\ge
	2^i \, \frac{v_1}{\det^1(\iota(\Sigma_1))}\,d^{r-1}+o(d^{r-1})\]
	for all $i$ and $d$. 
	So $H_\Sigma(d)$ has faster growth than $C\,d^{r-1}$ for every constant $C$, as otherwise we would have $C\ge 2^iv_1/\det^1(\iota(\Sigma_1))$ for all $i>0$, a contradiction.
	
	Finally, if $\rrk\Sigma=\infty$, one may find a sequence of subsemigroups $\Sigma_{1}\subset\dots\subset\Sigma_r\subset\dots\subset\Sigma$ with each $\Sigma_r$ finitely generated and of rational rank $r$. Then $H_\Sigma(d)\ge H_{\Sigma_r}(d)=\vol{}{\Delta(\Sigma_r)}d^{r-1}+o(d^{r-1})$ for every $r$, and the last claim follows.
\end{proof}

\begin{definition}
	We define the \emph{volume} of a linearly bounded semigroup $\Sigma$ such that $\Sigma_\Z$ is finitely generated as $\vol{}{\Sigma}=\vol{}{\Delta(\Sigma)}$. 
	It is a positive real number satisfying $H_\Sigma(d)=\vol{}{\Sigma}\,d^{\rrk \Sigma-1}+o(d^{\rrk \Sigma-1})$.
\end{definition}

\begin{definition}[Abstract restricted semigroup]
	Given a graded, cancellative, torsion free semigroup $\Sigma$, let $\Sigma_\Q\deq \Sigma_\Z\otimes_\Z\Q$, and let $\iota_\Sigma:\Sigma\hookrightarrow \Sigma_\Q$ be the canonical embedding. 
The grading of $\Sigma$ extends to a linear form $\deg:\Sigma_\Q\rightarrow\Q$, and we denote $L_d=\deg^{-1}(d)\subset \Sigma_\Q$ the degree $d$ hyperplane.
For everly linear subspace $W\subset L_0$ and every $\sigma\in\Sigma$, we call $\Sigma|_{W+\sigma}=\Sigma\cap\iota_\Q^{-1}(\langle W+\iota_\Q(\sigma)\rangle_\Q)$ an \emph{abstract restricted semigroup}.
\end{definition}

\begin{definition}[Closed concave envelope, \cite{Roc70}, Section 7]\label{def:concave-envelope}
	Let $\Delta\subset V$ be a compact convex set, $f\colon \breve\Delta\to \RR$ a bounded real-valued function defined on a dense subset $\breve\Delta\subset\Delta$. The \emph{closed concave envelope} $f^c:\Delta\rightarrow\RR$ of $f$ is defined as 
	\[
	f^c(x) \deq \inf \{ g(x)\mid g\colon \Delta\to\RR \text{ is concave and upper-semicontinuous and } g|_{\breve\Delta}\dgeq f\, \}\ .
	\]	
\end{definition}

It follows from the definition that $f^c$ is itself concave and usc, in particular it is continuous in the interior of $\Delta$ and and along any line segment contained in $\Delta$. 
As we will see later, continuity along the boundary does not hold in general. 

\begin{ntn}\label{ntn:restricted} Let $\Sigma$ be a graded, cancellative, torsion free semigroup, and $W\subset L_0\subset \Sigma_\Q$ a linear subspace of dimension $s$ such that $\Sigma_W=W\cap\Sigma_\Z$ is a finitely generated group.
	By Theorem \ref{thm:nob-semigroup-general}, $\vol{}{\Sigma|_{W+\sigma}}$ is a finite real number, and $H_{\Sigma|_{W+\sigma}}(d)=\vol{}{\Sigma|_{W+\sigma}} d^{s-1}+o(d^{s-1})$.
	Fix an embedding $\iota:\Sigma \hookrightarrow V$ into a real vector space endowed with a volume form, and denote by the same symbol $\iota$ its unique extension to $\Sigma_\Q$. 
	Further denote $\overline{V}=V/\langle\iota(W)\rangle_\R$, $\bar{\iota}:\Sigma/W\rightarrow\overline{V}$ the embedding induced by $\iota$, and let $p:V\rightarrow \overline{V}$ be the projection.
\end{ntn}

\begin{definition}[Restricted volume function]
		Assume that $\iota$ satisfies $\dim_\R \langle\iota(W)\rangle_\R=s=\dim W$. In this case, the fiber of the map
	\begin{align*}
	\pi:\Sigma &\longrightarrow \Delta(\bar{\iota}(\Sigma/W))\\
	\sigma &\longmapsto \bar\iota(p(\sigma))/\deg(\sigma)
	\end{align*}
	over any point $\mathbf{v}=\pi(\sigma) \in \pi(\Sigma)$ is exactly the restricted semigroup $\Sigma|_{W+\sigma}$.
	If moreover $\iota(\Sigma)$ is asymptotically convex (Definition \ref{def:asymptotically-convex}), then the restriction of $\iota$ to $\Sigma|_{W+\sigma}$ is a full embedding and every $\Sigma/W$-rational point in the interior of $\Delta(\bar\iota(\Sigma/W))$ belongs to $\pi(\Sigma)$. 
	Thus we may define a map $f$ on the dense subset $\pi(\Sigma)$ of $\Delta(\bar\iota(\Sigma/W))$ as
	\[f(\pi(\sigma))=\vol{}{\Sigma|_{W+\sigma}}^{1/s}=
	\lim_{d\to\infty}\frac{H_{\Sigma|_{W+\sigma}}(d)^{1/s}}{d}\]
	whose closed concave envelope, raised to the $s$-th power, we call \emph{restricted volume function} of $\Sigma$ with respect to $W$, denoted $\operatorname{vol}_{\Sigma|W}\deq (f^c)^s:\Delta(\bar\iota(\Sigma/W))\rightarrow\R$. 
\end{definition}	

We saw in Theorem \ref{thm:nob-semigroup-general} that the volume of the Newton--Okounkov body of $\Sigma$ is not connected with its Hilbert function if $\Sigma_\Z$ is not finitely generated (in that case the volume even depends on the choice of a full embedding $\iota:\Sigma\hookrightarrow V$). 
Nevertheless, applying Theorem \ref{thm:restricted_body_semigrp} combined with Theorem \ref{thm:nob-semigroup-general} for the restricted semigroups, we can now prove that the volume of $\Delta(\iota(\Sigma))$ is the integral of the restricted volume function (which we stress is determined by the Hilbert function of the restricted semigroups) with respect to any subspace $W$ and any embedding $\iota$ satisfying the properties above.

\begin{corollary}\label{cor:volume-and-slices}
	Let $\Sigma$ be a graded, cancellative, torsion free semigroup, and $W\subset L_0\subset \Sigma_\Q$ a linear subspace of dimension $s$ such that $\Sigma_W=W\cap\Sigma_\Z$ is a finitely generated group.
	
	Let $\iota:\Sigma \hookrightarrow V$ be an embedding into a real vector space endowed with a volume form, and denote by the same symbol $\iota$ its unique extension to $\Sigma_\Q$. 
	If	$\iota(\Sigma)$ is asymptotically convex, linearly graded and linearly bounded, and $\dim_\R \langle\iota(W)\rangle_\R=s$,  then 
	\[\frac{\vol{}{\Delta(\iota(\Sigma))}}{\det^1(\iota(\Sigma_{W}))}=\int_{\Delta(\bar{\iota}(\Sigma/W))} \vol{\Sigma|W}{\mathbf{v}} d\mathbf{v} .\]
\end{corollary} 

\begin{proof}
	In the notations \ref{ntn:restricted}, the image of $\Delta(\iota(\Sigma))\subset L_1$ by the projection $p$ is $\Delta(\bar{\iota}(\Sigma/W))\subset \overline{L}_1$, where $\overline{L}_1$ denotes the hyperplane of vectors of degree 1 in $\overline{V}$, and vor every vector $\mathbf{v}\in\langle\Sigma\rangle_\Q$ in the relative interior of $\Delta(\Sigma/W)$, the slice $\Delta(\Sigma)\cap L_{\mathbf{v}}$  is the Newton--Okounkov set of the restricted semigroup $\Sigma|_{W+\sigma}$ for every  $\sigma \in \pi^{-1}(\mathbf{v})$, by Theorem \ref{thm:restricted_body_semigrp}.
	By Remark \ref{rmk:brunn-minkowski} there is a concave upper semicontinuous function defined on $\Delta(\Sigma/W)$ which agrees with $f$ on $\pi(\Sigma)$, therefore by definition the restricted volume function agrees with $f^r$ on $\pi(\Sigma)$. 
	Now by Theorem \ref{thm:nob-semigroup-general} applied to the restricted semigroups,	for $\sigma\in \Sigma$, $\mathbf{v}=\pi(\sigma)$ we have  
	\begin{equation}\label{eq:volume-slices}
			\vol{}{\Sigma|_{W+\sigma}} = \vol{\Sigma|W}{\mathbf{v}}= \frac{\vol{}{L_{\mathbf{v}}\cap\Delta(\iota(\Sigma))}}{\det^1(\iota(\Sigma_{W}))}\ ,
	\end{equation}
	and the claim follows.
\end{proof}

Corollary \ref{cor:volume-and-slices} can be seen as a far-reaching generalization of Kaveh--Khovanskii's theorem \ref{thm:volNOB_semigroup_KK}, which corresponds to the case $W=\Sigma_\Q$.
We will see in section \ref{sec:concave-transforms} that the volume formula of the Boucksom--Chen filtered Newton--Okounkov body \cite{BC} is another particular case of \ref{cor:volume-and-slices}.

\subsection{Ordered semigroups}\label{sec:ordered-semigroups}
We now turn our attention to graded semigroups endowed with a compatible total order; for these, it is natural to restrict the allowed embeddings into real vector spaces by rquiring that the order be preserved. Doing so, their Newton--Okounkov bodies are determined up to the action of a smaller group, and hence additional features of $\Delta(\Sigma)$ are invariants of $\Sigma$. 
Note that the Newton--Okounkov bodies used in algebraic geometry do correspond to ordered graded semigroups (see Lazarsfeld--Musta\c{t}\u{a} \cite{LM09}, Kaveh--Khovanskii \cite[Part III]{KK12}).
 
Let us now recall the main properties of such semigroups. 

\begin{definition}
	A group $\Gamma$ or a semigroup $\Sigma$
	is said to be \emph{ordered} if it is endowed with a total
	order $\le$ compatible with the operation,
	in the sense that $\gamma_1<\gamma_2 \Rightarrow
	\gamma_1+\gamma_3<\gamma_2+\gamma_3$
	for every $\gamma_3$. 	 
\end{definition}

If $\Sigma$ is an ordered commutative semigroup, then it
is cancellative, and its order can be extended uniquely to the group $\Sigma_\Z$,
endowing it with the structure of an ordered abelian group.
In particular, $\Sigma_\Z$ is torsion free.

Given any element $\gamma$ in an ordered abelian group $\Gamma$, 
we denote $|\gamma|=\gamma$ if $\gamma\ge 0$,
and $|\gamma|=-\gamma$ otherwise.
An ordered abelian subgroup $K\subset \Gamma$ is called 
\emph{isolated} if, given any $\gamma\in K$, 
$K$ contains every $\eta$ such that $|\eta|\le \gamma$.
The kernel of any order-preserving homomorphism between
ordered abelian groups is isolated; the quotient by
an isolated subgroup inherits a natural total order:
$\gamma +K \le \eta +K$ if $\gamma \le \eta$.

The set of isolated subgroups of an ordered abelian group
is totally ordered by inclusion.
The ordinal type of the set of proper isolated subgroups of $\Gamma$ is called the order rank of $\Gamma$, denoted $\rank_{\le}\Gamma$. 
For the basic example $\Gamma=\R^r_{\lex}$, the order rank is $r$ and the chain of isolated subgroups 
$0=K_1\subsetneq K_2 \subsetneq \dots \subsetneq K_r \subsetneq \Gamma$
has $K_i=\{0\}^{r-i+1}\times\R^{i-1}_{\lex}  \subset \R^r_{\lex}$.
We define the order rank of an ordered commutative semigroup $\Sigma$
as the order rank of $\Sigma_\Z$.

\begin{lemma}\label{lem:oag_rkn}
	With notation as above, 
	\begin{enumerate}
		\item 	an ordered abelian group $\Gamma$ is of order rank $r\in\N$ 
		if and only if it is isomorphic to an
		ordered abelian subgroup of $\R^r_{\operatorname{lex}}$.
		(This is Hahn's embedding theorem \cite[p. 62]{FS01} in the case 
		of finite rank).
		\item \label{it:chain}
		If $\Gamma$ is an
		ordered abelian subgroup of $\R^r_{\operatorname{lex}}$
		of order rank $r$
		and $0=K_1\subsetneq K_2 \subsetneq \dots 
		\subsetneq K_r \subsetneq \R^r_{\operatorname{lex}}$
		is the chain of isolated subgroups of $\R^r_{\operatorname{lex}}$,
		then 
		\[0=K_1\cap \Gamma \subsetneq K_2 \cap \Gamma\subsetneq \dots 
		\subsetneq K_r \cap \Gamma\subsetneq \Gamma \]
		is the chain of isolated subgroups of $\Gamma$.
\end{enumerate}\end{lemma}

If $0=K_1\subsetneq K_2 \subsetneq \dots \subsetneq K_r \subsetneq \Gamma$ is the chain of
isolated subgroups of the order rank $r$ ordered abelian group $\Gamma$,
the quotients $R_1=K_2/K_1$, $R_2=K_3/K_2, \dots, R_r=\Gamma/K_r$ are called \emph{components} of $\Gamma$. 
Each component is itself an ordered abelian group of order rank 1, so it can be embedded in $\R$.

\begin{lemma}\label{lem:ker-deg-isolated}
	If $\Gamma$ is an ordered abelian group, and 
	$\deg$ is an order-preserving grading of $\Gamma$,
	then $\ker(\deg)$ is the maximal proper isolated subgroup
	of $\Gamma$, and the top component of $\Gamma$ is $\Gamma/\ker(\deg)$,
	isomorphic to $\Z$.
\end{lemma}

\begin{proof}
	By hypothesis $\deg$ is order-preserving, hence $\ker(\deg)$ is an isolated subgroup.
	Moreover, by the assumption that gradings are surjective, $\Gamma/\ker(\deg)\cong\Z$.
	Since $\Z$ has no nontrivial isolated subgroups, it follows that $\ker(\deg)$ is maximal among the proper isolated subgroups of $\Gamma$.
\end{proof}

\begin{corollary}\label{cor:grading-by-first}
	Let $\Gamma$ be an ordered abelian subgroup of 
	$\R^r_{\lex}$ with an order-preserving grading
	$\deg:\Gamma\rightarrow \Z$. Then there exists a
	positive real number $t$ such that
	\begin{enumerate}
		\item $\Gamma$ is contained in the subgroup 
		$(\Z t \times \R^{r-1})_{\lex}$, and
		\item for all $(x_1,x_2,\dots,x_n)\in\Gamma$, $\deg(x_1,x_2,\dots,x_n)=x_1/t$.
	\end{enumerate}
\end{corollary}

In other words, every graded ordered subgroup of $\R^r_{\lex}$ is graded (up to a constant) by first component, and so by applying an appropriate automorphism of $\R_{\lex}^r$ we may assume that $\Sigma\subset (\Z\times\R^{r-1})_{\lex}$
and $\deg(\sigma)=x_1$ if $\sigma=(x_1,\dots,x_r)\in \Sigma$.

\begin{definition}
	An ordered abelian group $\Gamma$ with finite order rank is said to be \emph{discrete}
if each of its components is isomorphic to $\Z$, or equivalently, if it is isomorphic to $\Z^r_{\lex}$.		
We will say that an ordered semigroup $\Sigma$ is discrete if the ordered group $\Sigma_\Z$ is discrete. In that case rational rank and order rank coincide and we will call this number simply rank.
\end{definition}

\begin{remark}
	If $\Sigma$ is a discrete ordered semgroup, some of the embeddings $\Sigma\hookrightarrow\Z^r\subset\R^r_{\lex}$ are distinguished, namely those which respect the ordering.
	If $\iota:\Sigma\hookrightarrow\Z^r_{\lex}\subset\R^r_{\lex}$ is a full, order-preserving embedding, then by Lemma \ref{lem:oag_rkn}, the ordered subgroup $\langle\iota(\Sigma)\rangle_\Z\subset\R^r_{\lex}$ is a lattice, and composing with a suitable automorphism of $\R^r_{\lex}$ we may assume that in fact $\langle\iota(\Sigma)\rangle_\Z=\Z^r_{\lex}\subset\R^r_{\lex}$.
	This allows for a stronger version of Lemma \ref{lem:sl_invariance}:
\end{remark}

\begin{lemma}\label{lem:ordered_sl_invariance}
	Let $\Sigma$ be a graded discrete ordered semigroup of rank $r$. 
	Let $\iota_1,\iota_2:\Sigma\rightarrow \R^r_{\lex}$ be two order preserving embeddings,
	and denote $\Delta_1(\Sigma)$, $\Delta_2(\Sigma)$
	the respective Newton--Okounkov bodies. 
	Assume that $\langle\iota_j(\Sigma)\rangle_\Z=\Z^{r}_{\lex}\subset \R^{r}_{\lex}$ for both $j=1, 2$.
	Then
	\begin{enumerate}
		\item $(\iota_i)_1(\sigma)=\deg(\sigma)$ for every 
		$\sigma\in \Sigma$, i.e., for both embeddings, $\Sigma$ is
		graded by first component.
		\item There is a unipotent automorphism $\phi\in \SL_{r-1}(\Z)$ of $\Z^{r-1}$ represented by a lower triangular matrix such that $\phi(\Delta_1(\Sigma))=\Delta_2(\Sigma)$.
	\end{enumerate}
\end{lemma}

\begin{proof}
	The first claim follows from Lemma \ref{lem:ker-deg-isolated}.
	For the second, after lemma \ref{lem:sl_invariance}, we just need to check that if $\phi:\Z^{r-1}_{\lex}\rightarrow\Z^{r-1}_{\lex}$ is order-preserving then the matrix $A_\phi$ representing it is unipotent and lower triangular.
	Both properties follow from the fact that $\phi$ must preserve each isolated subgroup $\{0\}^{r-i-1} \times \Z^{i}_{\lex}$, $i=0, \dots, r-2$
	and the set of positive elements.
\end{proof}

\begin{remark}
	Let $\Sigma$ be a graded discrete ordered semigroup of rank $n$, whose associated ordered group $\Sigma_\Z$ has chain of isolated subgroups  $0=K_1\subsetneq K_2 \subsetneq \dots \subsetneq K_r \subsetneq K_{r+1}=\Sigma_\Z$, 
	and let $\iota:\Sigma \rightarrow\R^r_{\lex}$ be an order-preserving embedding. 
	By Lemma \ref{lem:oag_rkn}, $\iota$ descends to an order preserving embedding $\iota_i: \Sigma/K_i \rightarrow \R^{r-i+1}_{\lex}$.
	Moreover $\iota$ extends to $\iota:\Sigma_\Z\rightarrow \R^r_{\lex}$ and we denote $W_i=\langle {\iota}(K_i)\rangle_\R=\{0\}^{r-i+1}\times\R^{i-1}_{\lex}\subset L_0=\{0\}\times\R^{r-1}_{\lex}$ 
	and $\Sigma|_{K_i+\sigma}=\Sigma|_{W_i+\sigma}$, which \emph{is independent on the embedding.}	
\end{remark}

\begin{lemma}If $\Sigma$ is a graded ordered semigroup of order rank $r$, whose associated ordered group $\Sigma_\Z$ has chain of isolated subgroups  $0=K_1\subsetneq K_2 \subsetneq \dots \subsetneq K_r \subsetneq K_{r+1}=\Sigma_\Z$, and $i\in\{1,\dots,r+1\}$ is such that $K_i$ is a discrete ordered semigroup of rank $i-1$, then  \begin{enumerate}
		\item For every $\sigma\in\Sigma$, and every $j\in\{i,\dots,r\}$, the restricted semigroup $\Sigma|_{K_j+\sigma}$ is a discrete ordered semigroup of rank $j$.
		\item There is an ordered group automorphism $\phi$ of $\R^r_{\lex}$ such that $\phi\circ\iota(K_i)=\{0\}^{r-i+1}\times\Z^{i-1}_{\lex}\subset\R^r_{\lex}$.
	\end{enumerate}
\end{lemma}

Therefore, if $\iota(\Sigma)$ is asymptotically convex (for instance, if $\Sigma$ is discrete, by Khovanskii's Theorem \ref{thm:approx_semigroup}) then for every discrete isolated subgroup $K_i$ the hypotheses of Corollary \ref{cor:volume-and-slices} are satisfied,  te slices of $\Delta(\Sigma)$ in the direction $W_i=\langle \iota(K_i)\rangle_\R$ are the Newton--Okounkov bodies of its restricted semigroups, and their volume is defined.
Since the $K_i$ are canonically associated to $\Sigma$, the restricted volumes $\vol{}{\Sigma|_{K_i+\sigma}}$ are canonically associated to each $\sigma\in \Sigma$, and these invariants can be computed as volumes of slices of the Newton--Okounkov body, by \eqref{eq:volume-slices}.

\subsection{Valuations and Newton--Okounkov bodies of graded algebras}

Newton--Okounkov bodies of graded algebras are defined via valuations, used as a means of associating, to an appropriate graded algebra, a graded semigroup with the same Hilbert function. 
In this subsection we recall some basics of valuation theory for convenience of the reader and to fix notations, and we give a generalization of the construction of Newton--Okounkov bodies of $\Z$-graded linear series to more general graded algebras.

\begin{dfn}
	A \emph{valuation} on a field $K$ is a map $v:K^\times\rightarrow \Gamma$, where
	$\Gamma$ is an ordered abelian group, satisfying the following properties:
	\begin{enumerate}
		\item $v(fg)=v(f)+v(g), \, \forall f, g\in K^\times$,
		\item $v(f+g)\geqslant \min(v(f),v(g)), \, \forall f, g\in K^\times$,
	\end{enumerate}
	If $R\subset K$ is a subring such that $v(a)=0,  \, \forall a \in R\setminus 0$, we say that $v$ is an $R$-valuation.
	$v(K^\times)$ is called the \emph{value group} of the valuation.
	Two valuations $v, v'$ with value groups $\Gamma, \Gamma'$ respectively are said to be
	\emph{equivalent} if there is an isomorphism $\iota:\Gamma\rightarrow \Gamma'$ of ordered groups such that $v'=\iota\circ v$.
\end{dfn}

The subring
$$R_v=\{f\in K^\times\,|\,v(f)\geqslant 0\}\cup\{0\}$$
is a \emph{valuation ring}, i.e., for all $f\in K$,
if $f\not\in R_v$  then $f^{-1}\in R_v$;
its unique maximal ideal is $\M_v=\{f \in K \,|\, v(f)>0\}$ and
the field $K_v=R_v/\M_v$ is called the \emph{residue field} of $v$.
Two valuations $v, v'$ are equivalent if and only if
$R_v=R_{v'}$ \cite[VI, \S8]{ZS75II}.

If $R\subset K$ is a subring contained in the valuation ring $R_v$, then $v(R\setminus 0)$ is a subsemigroup of the value group, which will be extremely relevant in the sequel.

\begin{dfn}
	The \emph{rank} (respectively \emph{rational rank}) of a valuation $v$  is the order rank (respectively \emph{rational rank}) of its value group.  
	Hence the rational rank is at least as large as the order rank.
	The standard example of a valuation with rational rank larger then its rank is in \cite[VI, \S14, Example 1, p. 100]{ZS75II}.
	\end{dfn}

\begin{rem}\label{rem:rational-rank} If $k\subset K$ is a subfield with finite transcendence degree $n=\trdeg_k(K)$, the rational rank (and hence the rank) of every $k$-valuation on $K$ is bounded by $n$, and in fact, by the Zariski-Abhyankar theorem, $\trdeg_k(K_v)+\rrk(v)\le\trdeg_k(K)$. Thus, for every valuation of rational rank $n$, $K_v$ is algebraic over $k$. Moreover every valuation of rank $n$ is discrete, i.e., it has a  value group  isomorphic to $\Z^{n}_{\lex}$, and every valuation of rational rank $n$ has a value group isomorphic (as an abstract group, but not necessarily as an ordered group) to $\Z^n$
	(see \cite[VI, \S10 and \S14]{ZS75II}, \cite{Abh56}, or \cite[Chap. 6, \S10, n. 3, Corollaire 1, page 161]{Bou06}).
\end{rem}
 
\begin{definition}[Newton--Okounkov body of a $\Gamma$-graded algebra]\label{def:NOB-algebra}
	Let $k$ be an algebraically closed field and $K/k$ an extension with finite transcendence degree $n$, and let $v$ be a $k$-valuation of $K$, of maximal rational rank $n$.
	Denote $\Gamma_v=v(K^\times)$ the value group of $v$.
	Let $\Gamma$ be a graded (abelian) group and $R$ a graded $k$-subalgebra of the group algebra $K[\Gamma]$:
	\[R \subset K[\Gamma]\deq\bigoplus_{\gamma\in \Gamma} K u^\gamma,\]
	where $u$ is a dummy variable and $u^\gamma \cdot u^{\gamma'}\deq u^{\gamma+\gamma'}$. 
	Denote as usual $R_\gamma=R\cap Ku^\gamma$ (the assumption that $R$ is a \emph{graded} subalgebra means that $R=\bigoplus R_\gamma$). 
	$R$ is a domain, and therefore its support
	$\Supp R=\{\gamma\in\Gamma\,|\,R\cap Ku^\gamma \ne 0\}$
	is a (graded) semigroup. 
	Moreover, it is easy to see that
	\[ \Sigma_{v}(R)=\{(\gamma,\gamma_v)\in \Gamma\times\Gamma_v \,|\,
	\exists s\in K^\times, su^\gamma \in R_\gamma, v(s)=\gamma_v \}  \]
	is a subsemigroup of $\Gamma\times\Gamma_v$. 
	The group $\Gamma\times\Gamma_v$ is graded by $\deg(\gamma,\gamma_v)=\deg(\gamma)$.
	If $\Gamma\times\Gamma_v$ is finitely generated, or a fixed embedding $\iota:\Gamma\times\Gamma_v \hookrightarrow V$ is given, we define the \emph{Newton--Okounkov set of $R$ with respect to $v$} as the convex set determined by this semigroup: $\Delta_{v}(R)\deq\Delta(\iota(\Sigma_{v}))$.
	
	The points in $\Delta_v(R)$ of the form $\sigma/\deg(\sigma)$ for some $\sigma\in \Sigma_{v}(R)$, or equivalently of the form $(\gamma,v(s))/\deg(\gamma)$ for some $su^\gamma \in R_\gamma$, are called \emph{valuative}.
\end{definition} 
 
\begin{remark}
	The usual definition of Newton--Okounkov bodies of algebras as given by \cite{KK12}, for instance, coincides with the particular case of \ref{def:NOB-algebra} in which $\Gamma=\Z$.
	Note that, since $\Gamma$ is always assumed to be graded, every algebra $R$ as in definition \ref{def:NOB-algebra} is also $\Z$-graded, with homogeneous pieces $R_d=\bigoplus_{\deg\gamma=d}R_\gamma$; when we refer to the Hilbert function of $R$ we always mean the Hilbert function with respect to the $\Z$-grading: $H_R(d)=\dim R_d$.
\end{remark} 
 
 \begin{remark}
 	 The valuation $v : K\rightarrow \Gamma_v$ can be trivially extended to $K[\Gamma]$ by setting $v(u)=0$ and $v(\sum s_\gamma u^\gamma)=\min_\gamma(v(s_\gamma))$. 
 	 We denote the restriction to $R$ of this trivial extension $v_R: R\rightarrow \Gamma$, or if no confusion seems likely, simply $v$.
 \end{remark}
 
\begin{theorem}\label{thm:Hilbert-function-algebra-semigroup}
	Let $k$ be an algebraically closed field and $K/k$ an extension with finite transcendence degree $n$, and let $v$ be a $k$-valuation of $K$, of maximal rational rank $n$.
	Denote $\Gamma_v=v(K^\times)$ the value group of $v$.
	Let $\Gamma$ be a graded (abelian) group with finite rational rank $r$, and $R$ a $k$-subalgebra of the group algebra $K[\Gamma]$.
	Then the Hilbert functions of the algebra $R$ and $\Sigma_{v}(R)$ agree.
	In particular, $\lim_{d\to\infty} H_R(d)/d^{n+r-1}$ exists, it is finite if and only if $\Gamma\times\Gamma_v$ is finitely generated and $\Sigma_v(R)$ is linearly bounded, and in that case $\lim_{d\to\infty} H_R(d)/d^{n+r-1}=\vol{}{\Delta_{v}(R)}$.
\end{theorem} 
\begin{proof}
	As in the case $\Gamma=\Z$ considered in \cite{KK12}, from the hypothesis that $k$ is algebraically closed and Remark \ref{rem:rational-rank} it follows that 
	$\dim_k R_\gamma= |v_R(R_\gamma\setminus\{0\})|.$
	Hence $\dim_k R_d= \sum_{\deg\gamma=d} \dim_k R_\gamma=|(\Sigma_{v}(R))_d|$, i.e., the Hilbert functions agree. 
	The remaining claims follow from Theorem \ref{thm:nob-semigroup-general}.
\end{proof}
 
 \subsection{Newton--Okounkov bodies of Cartier divisors and of line bundles.}
Now we recall the construction of Newton--Okounkov bodies in the geometric case. 
This subsection is mainly expository, to fix notation and to show how the theory of \cite{LM09} fits in our context.
We believe that Example \ref{exa:negative-no}, on \emph{negative} Newton--Okounkov bodies, and Remark \ref{rem:slices-geometric}, on the variation of the body when a different but equivalent valuation is used, are new.
 
Assume $X$ is a normal projective variety of dimension $n$ over the algebraically closed field $k$, let $K=K(X)$ be the field of rational functions on $X$, and let $v$ be a $k$-valuation of $K$ of maximal rational rank $n$.
By the valuative criterion of properness \cite[II, 4.7]{HAG}, since $X$ is projective, there is a (unique) morphism
$$\sigma_{X,v}:\Spec (R_v)\rightarrow X$$
which, composed with
$\Spec (K(X)) \rightarrow \Spec( R_v)$, identifies $\Spec (K(X))$
as the {generic point} of $X$.
The image in $X$ of the closed point of $\Spec (R_v)$
(or the irreducible subvariety which is its closure)
is called the \emph{centre} of $v$ in $X$, and we denote it by $\cent_X(v)$, or simply $\cent(v)$ when the variety $X$ is understood.
In the case of a valuation of maximal rational rank, the residual field is $k$ by Remark \ref{rem:rational-rank}, and hence the center is actually a closed point of $X$. 

Let us recall the most used valuations of maximal rational rank.

\begin{exm}(Valuation associated to an admissible flag)\label{exa:flag} A \emph{full flag} $Y_\bullet$ of irreducible subvarieties
	\begin{equation}\label{eq:flag2}
	X = Y_0 \supset Y_1 \supset \ldots\supset Y_{r-1} \supset Y_n
	\end{equation}
	is called \emph{admissible}, if $\codim_X (Y_i)=i$ for all $0\leqslant i\leqslant\dim (X)=n$, and $Y_i$ is normal and smooth at the point $Y_n$, for all $0\leqslant i\leqslant n-1$. The flag is called \emph{good} if $Y_i$ is smooth for all $i=0,\ldots,n$.
	
	Fix an admissible flag $Y_\bullet$.
	Let $f\in K(X)$ be a non-zero rational function, and set
	\[
	v_1(f) = \ord_{Y_1}(f) \quad \text{and}   \quad  f_1 = \left.\frac{f}{g_1^{v_1(f)}}\right|_{Y_1}
	\]
	where $g_1=0$ is a local equation of $Y_1$ in $Y_0$ in an open Zariski subset around the point $Y_r$.  Continuing this way via
	\[
	v_i(f) = \ord_{Y_i}(f_{i-1})\ ,\  f_i = \left.\frac{f_{i-1}}{g_i^{v_i(f_{i-1})}}\right|_{Y_i} \quad \text {for all}\quad i=2,\ldots, n,
	\]
	where $g_{i}=0$ is a local equation of $Y_i$ on $Y_{i-1}$ around $Y_n$,
	we arrive at a function
	\[
	f \mapsto v_{Y_\bullet}(f) = (v_1(f),\ldots, v_n(f))\ .
	\]
	One verifies that $v_{Y_\bullet}$ is a valuation of maximal rank, whose center is the point $Y_n$, and its value group is the discrete ordered group of rank $n$, namely $\Z^n_{\lex}$.
\end{exm}

\begin{exm}[Proper transform flag]\label{ex:proper transform flag}
	Fix an admissible flag $Y_\bullet$.	
	Let $\pi\colon \tX\to X$ be a proper birational morphism such that $Y_n$ is outside of the exceptional locus of $\pi$.  
	Then $\pi^*$ induces a $K$-algebra isomorphism $K(X)\to K(\tX)$.
	Let us write $\tY_\bullet$ for the flag where $\tY_i$ is the proper transform of $Y_i$ under $\pi$. Then $v_{\tY_\bullet}(\pi^*(f)) = v_{\ybul}(f)$ for a rational function $f$ on $X$. 
\end{exm}

\begin{remark} \label{thm:flag-valuation}
	It was proved in \cite[Theorem 2.9 and Remark 2.10]{CFKLRS} that for every $k$-valuation
	of $K(X)$ of maximal rank $n=\dim (X)$ there exist a proper birational morphism
	$\pi:\tilde X \rightarrow X$  and an admissible flag
	\[
	Y_\bullet : \tilde X= Y_0 \supset Y_1 	\supset \ldots \supset Y_n
	\]
	such that $v$ is equivalent to the valuation associated to $Y_\bullet$.
\end{remark}

\begin{example}
	Let $p\in X$ be a smooth point, and $(x_1,\dots,x_n)$ local coordinates around $p$, so that the completion of the local ring at $p$ is the power series ring in the coordinates: $\widehat{\mathcal{O}_{X,p}}\simeq k[[x_1,\dots,x_n]]$.
	Let $\Gamma$ be an ordered group of rank $n$, and $\gamma_1,\dots,\gamma_n\in\Gamma$ positive elements with $\langle\gamma_1,\dots,\gamma_n\rangle_\Z=\Gamma$.
	Define a valuation in $k[[x_1,\dots,x_n]]$ by
	\[v\left(\sum a_\alpha\textbf{x}^\alpha\right)=
	\min\{\alpha\cdot\gamma\deq\alpha_1\gamma_1+\dots\alpha_n\gamma_n | a_\alpha\ne0\}.\]
	One easily checks that the restriction of this valuation to $K(X)\hookrightarrow\mathcal{O}_{X,p}\hookrightarrow k[[x_1,\dots,x_n]]$ has rational rank $n$, and its center is the point $p$.
	
	Note that flag valuations are particular instances of this construction, taking $\Gamma=\Z^n_{\lex}$, $(x_1,\dots,x_i)$ to be local equations for $Y_i$ and $\gamma_i$ equal to the $i$th unit coordinate vector in $\Z^n_{\lex}$.
\end{example}

\begin{example}[Composite valuations]\label{exa:composite}
	Let $r$ be an integer $0\le r< n$; let
		\[
		Y_\bullet : \tilde X= Y_0 \supset Y_1 	\supset \ldots \supset Y_r
		\]
	be a \emph{partial admissible flag}, i.e., 	$\codim_X (Y_i)=i$ for all $0\leqslant i\leqslant\dim (X)=r$, and $Y_i$ is normal and smooth at the point $Y_r$, for all $0\leqslant i\leqslant r$; 
	and let $\bar v$ be a valuation on $K(Y_r)$ of maximal rational rank $n-r$ and value group $\Gamma$.
	Then defining $v_i(f)$ and $f_i$ as in Example \ref{exa:flag} for $i=1, \dots, r$, and $v_{r+1}(f)=\bar v(f_r)$ one gets a valuation
	\[f \mapsto v_{Y_\bullet,\bar{v}}(f)=(v_1(f),\dots,v_{r+1}(f))\in\Z^r\times\Gamma\]
	of rational rank $r+n-r=n$.
\end{example}

\begin{rem}\label{rem:valuation of sections of subsheaves}
Global sections of subsheaves of $K(X)$ are evaluated naturally, in particular, $v(s)$ for $s\in H^0(X,\OO_X(D))$ is naturally defined for a Cartier divisor $D$. The situation with invertible sheaves is less comfortable, as the following example shows. Consider $X\deq \PP^1$, and $v\deq \ord_P$ for $P\in\PP^1$ an arbitrary point, let $Q\neq P\in \PP^1$. 
Then 
\[
H^0(X,\OO_X(Q)) \equ \{ f\in K(X) \mid \ddiv (f) + P \succcurlyeq 0  \} \cup \{0\}\ ,
\]
in particular $v(s)\geq 0$ for all $0\neq f\in H^0(X,\OO_X(Q))$. At the same time the isomorphic invertible sheaf $\OO_X(P)$ has a global section $f$ with $v(f) =-1$; one can in fact obtain global sections with arbitrarily negative valuations by considering $\OO_X(mP-(m-1)Q)$. 

We see that in order to be able to define valuations of global sections of invertible sheaves some choices are needed.
\end{rem}

\begin{rem}\label{rem:val_divisors}
Let $L$ be an invertible sheaf on $X$, $v$ a valuation of $K(X)$. Let $D$ be a Cartier divisor on $X$ such that $L\simeq \OO_X(D)$ and the center of $v$ is not contained in $\Supp D$. Write $\phi_D\colon L\to \OO_X(D)$ an isomorphism. For a global section $0\neq s\in H^0(X,L)$ we set 
\[
v(s) \deq v(\phi_D(s)) \dgeq 0 \ .
\] 
Observe that the value $v(s)$ is independent of the choice of $\phi_D$. Moreover, if $D'\sim D$ is another Cartier divisor, then $v(\phi_D(s)) = v(\phi_{D'}(s)) + v(f)$  upon writing $D-D'=\ddiv (f)$ for a suitable rational function $f\in K(X)$,
hence by assuming $\cent (v)\not\subseteq\Supp D'$ as well, we obtain that $v(f)=0$, and that $v(\phi_D(s)) = v(\phi_{D'}(s))$. If $E=V(s)$ is a divisor then we define $v(E)$ to be $v(s)$.

We can conclude that by taking an isomorphism $L\stackrel{\sim}{\to} \OO_X(D)$ for  a Cartier divisor $D$ such that $\cent (v)\not\subseteq \supp(D)$, $v(s)$ has a well-defined non-negative value for every non-zero global section of $H^0(X,L)$. Note that this is the implicit convention used in \cite{LM09}. 
\end{rem}

Fix a projective variety of dimension $r$ over $k$, $K=K(X)$ its field of rational functions, and a (not necessarily effective) Cartier divisor $D$ on $X$. Write 
\begin{equation}\label{eq:section-ring}
R\deq R(X,\OO_X(D))=\bigoplus_{d\in \Z} H^0(X, \OO_X(dD))u^d\subset K[u,u^{-1}]=K[\Z] 
\end{equation}
for the section ring of $D$. 
Here, we consider the space $H^0(X, \OO_X(dD))$ as a space of rational functions on $X$ with singularities only along positive components of the divisor $D$, rather than as sections of the line bundles $\mathcal{O}_X(D)$. 
For each $d$ the graded  piece $R_d= H^0(X, \OO_X(dD))$ therefore comes equipped with a natural inclusion $R_d\subset K$, whence the inclusion $R\subset K[u,u^{-1}]$ of \eqref{eq:section-ring}. 
(The existence of this natural inclusion is the reason we initially work with Cartier divisors rather than line bundles.)  
In fact, since $dD$ and $-dD$ cannot be both linearly equivalent to effective divisors, either $R\subset K[u]$ or $R\subset K[u^{-1}]$.
Now if $v:K^\times\rightarrow \Gamma$ is a valuation of maximal rational rank equal to $n$,
we are in the situation of the preceding subsection, and we can
define the Newton--Okounkov body of $D$:

\begin{definition}[Newton--Okounkov bodies of section rings] \label{defn:NOb I}
	Let $X$ be a projective variety of dimension $n$, $D$ a Cartier divisor, $R\deq R(X,\mathcal{O}_X(D))$ the section ring of $D$ defined in \ref{eq:section-ring}, and $v$ a valuation of $K$ of maximal rational rank $n$.
	Then the Newton--Okounkov set $\Delta_v(D)$ of $D$ with respect to $v$ is defined as the Newton--Okounkov set of $R$ with respect to $v$.
	We also use the notation $\Sigma_v(D)=\Sigma_{v}(R)$ for the corresponding semigroup.
\end{definition} 
\begin{remark}\label{rem:counting}
	By Theorem \ref{thm:Hilbert-function-algebra-semigroup},
	$\dim H^0(X,\O_X(dD))=H_{\Sigma_v(D)}(d)$.
\end{remark}

The fact that $\Delta_v(D)$ is a body for $D$ big follows from the following basic result:

\begin{proposition}[Kaveh--Khovanski, Lazarsfeld--Musta\c{t}\u{a}]\label{pro:body}
	As a semigroup graded by first component, $\Sigma_v(D)$ is linearly bounded. If moreover $D$ is big, then $\Sigma_v(D)$ has rank ${n +1}$.
\end{proposition} 

\begin{proof}[Remarks on proofs of \ref{pro:body}]
	If $v$ is a flag valuation, the claims are proved by Lazarsfeld-Mustata 
	in \cite[Lemma 1.10 and Lemma 2.2]{LM09}.
	In characteristic zero, if the valuation $v$ has maximal rank $n$, then there is a birational model $\pi\colon\tilde X\rightarrow X$ and an admissible flag 
	$Y_\bullet$ on $\tilde X$ such that $v$ is the valuation
	determined by the flag $Y_\bullet$ (Remark \ref{thm:flag-valuation}). 
	Moreover, $\Sigma_v(D))=\Sigma_v(\pi^*(D))$, and the claims reduce to the flag case.
	The general case of a valuation of rank less or equal to $n$ is proved by
	Kaveh-Khovanskii \cite{KK12} and Boucksom \cite{Bou14}
	using the Hilbert-Serre Theorem. 
	We remark that the full strength of the Hilbert-Serre
	Theorem is not needed in those proofs, for which the
	existence of the volume
	$\vol{}{L}=\lim H^0(X,\O_X(dD))/(k^{d}/d!)$ as a real number
	is enough. Since every projective variety $X$ supports
	admissible flags, the Lazarsfeld-Mustata argument
	together with Theorem \ref{thm:Hilbert-function-algebra-semigroup} below suffice to 
	prove the existence of the volume. Hence 
	Proposition \ref{pro:body} can be proved without resource
	to Hilbert-Serre.
\end{proof}

\begin{rem}
	With notation as above, if $D'\sim D+\ddiv(f)$ for a rational function $f$ on $X$, then 
	\[
	\Delta_v(R(X,\mathcal{O}_X(D'))) \equ \Delta_v(R(X,\mathcal{O}_X(D))) + v(f)\ \subseteq \RR^{n}\ .
	\]
\end{rem}

\begin{exm}[Negative Newton--Okounkov bodies]\label{exa:negative-no}
	Let $X=\PP^1$, $P\neq Q$ point in $\PP^1$, and  $v=\ord_P\colon K(X)^\times\to\ZZ$ (cf. Remark~\ref{rem:valuation of sections of subsheaves}). Then
	\[
	\Delta_{\ord_P}(\OO_{\PP^1}(mP+(m-1)Q)) \equ [-m,-(m-1)]  \dsubseteq \RR^1
	\] 
	for every natural number $m$. 
	
	For an arbitrary variety $X$, the set $\{v(f), f\in K(X)\}$ is exactly the value group of $v$, so the bodies obtained in the previous remark by changing representatives $D'$ are exactly all the translates by integer vectors of $\Delta_{v}(R(X,\mathcal{O}_X(D)))$.
\end{exm} 

\begin{remark}
	It is natural to consider the case of graded subalgebras of $R(X,\mathcal{O}_X(D))$ as well (i.e., Newton--Okounkov bodies of graded linear series); we will briefly use this possibility in section \ref{sec:concave-transforms}. 
	This more general setup was studied, including conditions for the Newton--Okounkov set to be a body, by Lazarsfeld--Musta\c{t}\u{a} \cite{LM09}.
\end{remark}

Now we move on to the definition of Newton--Okounkov bodies for invertible sheaves as found in \cite{LM09} for instance. 
As observed in  Remark~\ref{rem:val_divisors}, the fact that the construction is well-defined and delivers a non-negative convex body relies on certain choices. 

\begin{definition}[Newton--Okounkov bodies of invertible sheaves]\label{defn:NOb II}
	Let $X$ be a projective variety, $L$ an invertible sheaf on $X$,  $v:K(X)^\times\rightarrow \Gamma$ a valuation of maximal rational rank equal to $n=\dim X$. Let $D$ be an arbitrary 
	Cartier divisor on $X$ such that $L\simeq \OO_X(D)$ and $\cent (v)\not\subseteq \Supp D$. The	Newton--Okounkov set $\Delta_v(L)$ is then defined to be $\Delta_v(D)$, where the latter is the convex set from \ref{defn:NOb I} (and it is a body if $L$ is big). 
\end{definition}

\begin{rem}
	It follows from Remark~\ref{rem:val_divisors} that $\Delta_v(L)$ is independent of the choice of $D$ and is contained in $\RR_{\geq 0}^n$. In particular $\Delta_{\ord_P}(\OO_{\PP^1}(1)) = [0,1]$ for an arbitrary point $P\in \PP^1$.
\end{rem} 

\begin{convention}[Newton--Okounkov bodies of invertible sheaves and line bundles]\label{conv:NOb}
	From now on when we talk about Newton--Okounkov bodies of invertible sheaves or line bundles we will mean the Newton--Okounkov bodies from Definition~\ref{defn:NOb II}. This is in line with \cite{LM09} and all subsequent research. 	
\end{convention} 

The main theorem of the theory of Newton--Okounkov bodies of line bundles, which we proceed to state, is then an immediate consequence of Theorem \ref{thm:nob-semigroup-general} and Remark \ref{rem:counting}.

\begin{theorem}[{\cite[Corollary 3.2]{KK12}}, {\cite[Theorem 2.3]{LM09}}]
	$\vol{X}{L}=n!\, \vol{\RR^n}{\Delta_v(L)}$.
\end{theorem}

\begin{remark}\label{rem:slices-geometric}
 As was seen in section \ref{sec:ordered-semigroups}, slices of $\Delta_v(L)$ corresponding to isolated semigroups are invariant under order isomorphisms of the value group of $v$, and they are Newton--Okounkov bodies of the corresponding restricted semigroups.
 Hence, equivalent valuations $v$, $v'$ of maximal rank determine Newton--Okounkov bodies $\Delta_v(L)$, $\Delta_{v'}$, whose ``vertical'' slices of all dimensions are equal up to unipotent $\Z$-linear isomorphisms.
 If $v$ hais a flag valuation, it was shown by Lazarsfeld-Musta\c{t}\u{a} that these ``vertical'' slices have a geometric meaning, namely they are Newton--Okounkov bodies of restricted linear series.
 On the other hand, for equivalent valuations $v, v'$ of rank smaller than their rational rank, vertical slices may differ, as the corresponding subgroups of the value group are not invariant by automorphisms of the value group.
\end{remark}

\begin{remark}[Global Newton--Okounkov body]\label{rem:global}
	Given $D_1, \dots, D_r$ effective divisors whose classes form a $\Z$-basis of $N^1(X)$, and a valuation of maximal rank $v$, the Newton--Okounkov body $\Delta_v(R)$ of the algebra
	\[R=\bigoplus_{d\in \Z^r} H^0(X,\mathcal{O}_X(d_1D_1+\dots +d_rD_r))
	\subset K[\Z^r]\]
	is essentially the global Newton--Okounkov body of \cite[Theorem 4.5]{LM09}, and the fact that its slices are the Newton--Okounkov bodies of all numerical divisor classes follows at once from Theorem \ref{thm:restricted_body_semigrp}.
\end{remark}

\section{Concave functions on Newton--Okounkov bodies}
\label{sec:concave-transforms}

Here we recall the construction of concave transforms of filtrations (also known as Okounkov functions), which yields an interesting class of examples of functions on Newton--Okounkov bodies. The first three subsections  are   mostly  expository, and follow the original works and \cite[Section 4]{KMSwB} quite closely. There exist two different points of view regarding the construction of such functions, due to Boucksom--Chen \cite{BC} via partial Newton--Okounkov bodies, and  Witt-Nystr\"om  \cite{WN14} using concave envelopes. The two give rise to the same function.
In subsection \ref{sec:ct-via-graded-linear-series} we show how the approach of \cite{BC} links with our results of section \ref{sec:convex-objects-semigroups}, and in particular that subgraphs of concave transforms are Newton--Okounkov bodies of suitable semigroups.

\subsection{Filtrations}

Let $A$ be a ring and $R$ an $A$-algebra, and denote
$f:A\rightarrow R$ the structure map.
Denote $\sigma(R)$ the set of additive subgroups of $R$, partially ordered by inclusion.

\begin{dfn}
A filtration of $R$ indexed by the
ordered abelian group $\Gamma$ is an order-reversing map
\begin{align*}
\Gamma \overset{F}{\longrightarrow} & \,\sigma(R)\\
\gamma \longmapsto & \,F_\gamma
\end{align*}
Unless otherwise specified, in this work all filtrations $F$ will be 
\begin{enumerate}
\item \emph{complete}, i.e., $\bigcup_{\gamma \in \Gamma} F_\gamma=R$,
\item \emph{multiplicative}, i.e., $F_\gamma \cdot F_\eta \subseteq F_{\gamma+\eta}$ for all 
$\gamma,\eta \in \Gamma$,
\item \emph{$A$-filtrations}, i.e., $f(A)\subset F_0$.
\end{enumerate}
Usually $A=k$ will be the algebraically closed base field, and $R\subseteq K[\Gamma']$ will be a graded subalgebra of a group algebra.
\end{dfn}

\begin{rem}
If $F$ is a $k$-filtration of $R$, then $F_0$ is a sub-$k$-algebra of $R$, and each $F_\gamma$ is a sub-$F_0$-module of $R$. 
\end{rem}

\begin{exm}[Filtrations by ideals] \label{ideal_filtrations}
The multiplicative $R$-filtrations of $R$ satisfy that $F_\gamma=R$ for all $\gamma\le 0$ and
	$F_\gamma$ is an ideal of $R$ for all $\gamma$. 
	As particular instances we have:
	\begin{enumerate}
        \item
        Given an ideal $I\subset R$, 
        \[F_t=
        \begin{cases}
        R &  \text{ if }t\le 0\\
        I & \text{ if }t\ge 1
        \end{cases}\]
        defines a filtration of $R$ by ideals, indexed by $t\in\Z$.
        \item 
        Given an ideal $I\subset R$, 
		\[F_t=
		\begin{cases}
		R &  \text{ if }t\le 0\\
		I^{t} & \text{ if }t\ge 1
		\end{cases}\]
        defines a filtration of $R$ by ideals, indexed by $t\in\Z$.		
		\item If $R$ is a domain, and 
		$v:R\setminus \{0\} \rightarrow \Gamma$ is a 
		nonnegative valuation on $R$, with value group $\Gamma$, then
		\[F_\gamma=
		\{a \in R | v(a)\ge \gamma \}
		\]
        defines a filtration of $R$ by ideals, indexed by $\Gamma$.		
	\end{enumerate}
\end{exm}

\begin{dfn}
	Given a filtration $F$ on $R$, indexed by $\Gamma$,  denote for every $\gamma\in\Gamma$
\[F_\gamma^{+}=\bigcup_{\eta > \gamma} F_\gamma.\]
Clearly $F_\gamma^{+}\subset F_\gamma$, and if $F$ is an 
$k$-filtration, then $F_\gamma^{+}$ is an $k$-vector subspace of $R$.
The quotients $\overline{F}_k=F_\gamma/F_\gamma^{+}$
will be called \emph{components} of the filtration.
The support of the filtration $F$ is the subset of $\Gamma$
defined as $\Supp F=\{\gamma\in \Gamma \mid \overline{F}_k\ne 0\}$.
\end{dfn}

In many cases of interest, such as the last two examples 
in \eqref{ideal_filtrations}, the support is a 
subsemigroup of $\Gamma$, but this is not always the case,
as shown by the first of the examples.

\begin{definition}\label{graded_filtration}
	Let $K=K(X)$ be the field of rational functions on some projective variety $X$ of dimension $n$, 
	let $R\subset K[\Gamma]$ be a $\Gamma$-graded $K$-algebra for some \emph{graded} abelian group $\Gamma$, and let $F$ be a filtration on $R$ indexed by the group $\Gamma_F$.
	Fix the notations 
	$ F_{\tilde{\gamma}}(\gamma)= F_{\tilde{\gamma}}\cap R_{\gamma}$,
	$ F_{\tilde{\gamma}}^+(\gamma)= F_{\tilde{\gamma}}^+\cap R_{\gamma}$, where $\gamma\in\Gamma$, $\tilde{\gamma}\in \Gamma_{F}$.
	We say that ${F}$ is a \emph{homogeneous} filtration if
	${F}_{\tilde{\gamma}}=\bigoplus_{\gamma\in\Gamma}  F_{\tilde{\gamma}}(\gamma)$
	for every $\tilde{\gamma}\in \Gamma_F$.
	We say that a homogeneous filtration ${F}$ on $R$ is \emph{linearly bounded} if the  semigroup generated by the \emph{graded support}
	\[ 
	\left\{ (\gamma,\tilde{\gamma})\in \Gamma\times\Gamma_F \,\left|
	\frac{{F}_{\tilde{\gamma}}(-\gamma)}
	{{F}_{\tilde{\gamma}}^+(-\gamma)} \ne 0\right.\right\}\]
	(which inherits the grading from $\Gamma$) is linearly bounded.
	
	In the next sections we will deal with the case $\Gamma=\Z$ (so $R$ is a subalgebra of the polynomial ring $K[\Z]\cong K[u]$) and $\Gamma_F\subset \R$.
	The \emph{jumping numbers} of a homogeneous filtration ${F}$ on $K[u]$ in degree $d$ are the following $N=\dim R_d$ real numbers:
	\[e_{\min} ({F}(d)) = e_N ({F}(d)) \le \dots \le e_1 ({F}(d)) = e_{\max} ({F}(d))\]
	defined by
	\[e_j ({F}(d)) \deq \sup\{t \in \Gamma_F\subset\R\,|\, \dim {F}_t(d) \ge j\} .\]
	The mass of ${F}(d)$ is $\operatorname{mass}({F}(d))\deq\sum e_j({F}(d))$, and its positive mass is $\operatorname{mass}_+({F}(d))\deq\sum_{e_j({F}(d))>0} e_j({F}(d))$.
	We also denote
	\[ e_{\min}({F})\deq \inf\left\{\frac{t}{d} \,|\, {F}_t(d)\ne R_d\right\}, \qquad
e_{\max}({F})\deq \sup\left\{\frac{t}{d} \,|\, {F}_t(d)\ne 0\right\}.
\]
	It is easy to see that these are finite real numbers if and only if $F$ is linearly bounded.
\end{definition}

Note that
\[e_{\min}( F)=\inf\left\{\frac{e_{\min}({F}(d))}{d}\right\}=
\inf\left\{\left.\frac{t}{d}\,\right| F_t(d)/{F}_t^+(d)\ne 0\right\}.\]
If $e_{\min}({F})=e_{\max}({F})$ then the filtration is trivial; we henceforth assume that $e_{\min}({F})<e_{\max}({F})$.

\begin{dfn}\label{def:bounded-rees}
	The \emph{Rees algebra} of the filtration $F$, indexed by $\Gamma_F$, is defined as the following graded subalgebra of $R[\Gamma_F]$:
	\[
	\Rees(F) = \bigoplus_{t \in \Gamma_F} F_t u^t 
	\subset R[\Gamma_F] \subset K[\Z\times \Gamma_{ F}].\]
	Every subsemigroup of $\Gamma_F$ determines a subalgebra of the Rees algebra; in the sequel we shall be interested in \emph{bounded} Rees algebras. 
	For every $B\in \R$, we define the \emph{Rees algebra of $F$ bounded by $B$} as
\[
\Rees_{B}( F) = \bigoplus_{\substack{t\in \Gamma_{{F}}\\t \ge Bd}}  F_t(d) u^t \subset \Rees(F)
\subset K[\Z\times\Gamma_{ F}].
\]
\end{dfn}

\begin{example}
	If $Z\subseteq X$ is a smooth subvariety contained in the smooth locus of $X$, then the $\ord_Z$ is a discrete valuation on $K$, and therefore it determines a filtration on $R$ as in Example  \ref{ideal_filtrations}, which is homogeneous and linearly bounded. As explained in \cite{WN} (see also \cite{RT1,RT2,Szek}), test configurations also give rise to such filtrations on section rings of ample (or at least big and nef) line bundles. 
\end{example}

\subsection{Definitions of concave transform}
The first approach to concave transform functions is the one taken in \cite{WN14} (see also \cite[Subsection 4.1]{KMSwB}), namely as concave envelopes, which can be used in concrete computations to some extent. The functions are defined in two steps, first on the dense set of valuative points in $\Delta_v(L)$, then on the whole of $\Delta_v(L)$ via convex geometry. 

\begin{defi}[Concave transform I]
With notation as above, let $\alpha\in \Delta_v(L)$  be a valuative point. We define
\[
\widetilde{\phi_{ F}}(v) \deq \lim_{d\to\infty} \frac{1}{d} \sup\{t\in\RR\mid \exists\, s\in  F_t(d) \text{ such that } v(s) = d\alpha\}\ .
\]
The existence of the limit follows from Fekete's Lemma \cite{Fekete} (cf. the Appendix in \cite{KMSwB}). The function $\widetilde{\phi}_{ F}$  is defined on a dense subset of $\Delta_v(L)\cap\QQ^n$, so we can use a concave envelope (Definition \ref{def:concave-envelope}) to pass to the whole Newton--Okounkov body. Thus, 
the \emph{concave transform} $\phi_{\Finv}\colon \Delta_v(L)\rightarrow\R$ of the filtration $ F$ on $\Delta_v(L)$ is defined to be the closed convex envelope $(\widetilde{\phi}_{ F})^c$. 
\end{defi}

\begin{rem}
Whenever we believe that it does not lead to confusion, we will use $\ord_Z$ for the concave transform of the filtration arising from the valuation $\ord_Z$ as well. 
Assume that $X$ is smooth, and $L$ is a big line bundle.
 A quick consequence of the definition is that 
\[
\inf_{\Delta_v(L)} \ord_Z \dgeq \ord_Z (\| L\| )\ ,
\] 
where the right-hand side is the asymptotic order of vanishing of $L$ along $Z$. Equality is not expected to hold in general (cf. \cite[Example 2.7]{KL_noninf}). 
\end{rem}

\noindent
Concrete examples computed via this definition can be found in \cite[Subsection 4.4]{KMSwB} for instance.

\medskip
\label{sec:ct-via-graded-linear-series}

The second approach, via graded linear series, is due to 
Boucksom--Chen \cite{BC} (see also \cite{BKMS}). With notation as so far, let $t\in\RR$ be arbitrary. and set 
\[
V_t(d) \deq  F_{td}(d) \ .
\]
It is immediate that $V_t(\bullet)$ form a graded subalgebra of $R(X,L)$, corresponding to a graded linear series $\bigsqcup_{d\ge 0} \P(V_t(d))$, and the Newton--Okounkov bodies $\Delta_v(V_t(\bullet))$ are a non-increasing collection of compact convex subsets of $\Delta_v(L)$. 

\begin{definition}[Concave transform II]
With notation as above, the \emph{concave transform} of the filtration $ F$ is defined to be 
\[
\phi_{ F} (\alpha) \deq \sup \{ t\in\RR\mid \alpha\in\Delta_v(V_t(\bullet)) \}\ .
\]
\end{definition}

\begin{remark}\label{rmk:concave-transforms}
	It is known that the two definitions of concave transforms agree (see \cite[Remark 1.10]{BC} and \cite[Lemma 4.9]{KMSwB}). We give below a semigroup-theoretic proof.
\end{remark}

\subsection{Properties of concave transforms}

We collect most of the known properties of concave transforms of filtrations, with special attention to filtrations given by order of vanishing along some subvariety. 

\begin{theorem}[Continuity of concave transforms, \cite{KMSwB}, Theorem 1.1,  \cite{BKMS}, Theorem B]
	\begin{enumerate}
		\item Let $X$ be an $n$-dimensional projective variety over $K$, $L$ a $\QQ$-effective line bundle on $X$, $v$ a valuation of $K(X)$ of rational rank $n$, $ F$ a linearly bounded filtration on $R(X,\mathcal{O}_X(L))$. If the Newton--Okounkov body $\Delta_v(L)$ is a polytope (not necessarily rational), then $\phi_{ F}$ is continuous on the whole of $\Delta_v(L)$. 
		\item There exists a projective variety $X$, a big line bundle $L$, an admissible flag $\ybul$ on $X$, and a divisorial valuation $\ord_Z$ of $K(X)$ such that the concave transform $\ord_Z$ is \emph{not} continuous on $\Delta_v(L)$.   
	\end{enumerate}
\end{theorem}

\begin{rem}
	According to \cite[Theorem B]{KLM}, the Newton--Okounkov body $\dyl$ will always be a polygon provided $X$ is a smooth surface. Hence concave transforms are always continuous in dimension two. 
\end{rem}

Concave transforms exhibit the formal properties expected of asymptotic invariants. 

\begin{theorem}[Formal properties]\label{thm:formal properties}
	With notation as above, 
	\begin{enumerate}
		\item $($Homogeneity$)$ For each $a\in\NN$ let $ F_a$ be the filtration defined by $( F_a)_{\gamma}= F_{a\gamma}$. Then $\phi_{ F_{a}}(a\alpha) \equ a\cdot \phi_{ F}(\alpha)$  for all $\alpha\in\Delta_v(L)$.  
		\item $($Numerical invariance $)$ If $L'\equiv L$ are two numerically equivalent big line bundles, then $\phi_{ F^L} \equ \phi_{ F^{L'}}$ as functions on $\Delta_v(L)=\Delta_v(L')$.  
	\end{enumerate}	
\end{theorem}
\begin{proof} 
	This is \cite[Theorem 4.14 and Proposition 5.6]{KMSwB}
\end{proof}

In lucky cases invariants of the functions $\phi_{ F}$ will not depend on  the domain $\Delta_v(L)$, or more precisely, the choice of $\ybul$ or $v$. In this case they give rise to asymptotic invariants of the line bundle $L$. 

\begin{theorem}[Local positivity invariants from concave transforms]\label{thm:local positivity invariants}
	Let $X$ be a smooth projective variety, $Z\subseteq X$ a smooth subvariety, $L$ a big line bundle on $X$, and $\ybul$ an arbitrary admissible flag on $X$. Then the numbers 
	\[
	\max_{\dyl} \ord_Z\ \ \ \text{and}\ \ \ \int_{\dyl} \ord_Z
	\]
	are independent of the choice of $\ybul$. 
\end{theorem}
\begin{proof}
	The first number is independent of the choice of the flag by \cite{DKMS2}, Proposition 2.2 and Theorem 2.4 (note that both proofs go through verbatim in the current setting), the second  one is \cite[Corollary 1.13]{BC}. 
\end{proof}

\begin{rem}
	From the proof of \cite[Proposition 2.2]{DKMS2} we see that 
	\[
	\max_{\dyl} \ord_Z \equ \mu(L;Z) \deq \sup \{ t\geq 0\mid \pi^*L-tE \text{ is pseudoeffective} \}
	\]
	where $\pi\colon Y\to X$ is the blowing-up of $X$ along $Z$ with exceptional divisor $E$, therefore the function $\ord_Z$ recovers a piece of the birational geometry of $X$. 
	On the other hand the integral of $\ord_Z$ is not known to be expressable in such terms.  
\end{rem}

\begin{definition}
	Let $X$ be a projective variety, $L$ a line bundle on $X$, $v$ a valuation on $K(X)$ of maximal rational rank, and $w$ a divisorial valuation on $K(X)$. We define 
	\[
	\iota(L;w) \deq \frac{1}{\vol{X}{L}}\cdot \int_{\Delta_v(L)} \phi_w\ .
	\]
	If $w$ is order of vanishing along a smooth subvariety $Z\subseteq X$ then we write $\iota(L;Z)$ for $\iota(L;w)$.
\end{definition}

\begin{lemma}
	With notation as above, if $\pi\colon \tX\to X$ is a proper birational morphism, then $\iota(\pi^*L;w)=\iota(L;w)$. 
\end{lemma}
\begin{proof}
	This is immediate from the definition. 	
\end{proof}

Using  this, we obtain a reasonably concrete formula for $\ilx$ via Fubini's theorem (cf. \cite[Theorem 2.24]{BKMS}).

\begin{proposition} \label{prop:integral formula}
	With notation as above, 
	\[
	\ilx \equ \frac{1}{\vol{\tX}{\pi^*L}} \cdot \int_{0}^{\infty} t\cdot \vol{\tX|E}{\pi^*L-tE}\, dt \ .
	\]
\end{proposition}
\begin{proof}
	Let $\pi\colon \tX\to X$ be a proper birational morphism for which $w=\ord_E$ for an irreducible prime divisor $E$. Consider an admissible flag $\ybul$ on $\tX$ with $Y_1=E$. 
	By Fubini's theorem and the Lazarsfeld--Musta\c t\u a slicing theorem we obtain 
	\begin{eqnarray*}
		\ilx  & \equ &   \frac{1}{\vol{X}{L}}\cdot \int_{\Delta_v(L)} \phi_w  \equ \frac{1}{\vol{\tX}{\pi^*L}}\cdot \int_{\Delta_{\ybul}(\pi^*L)} \phi_{\ord_E}   \\
		& = & \frac{1}{\vol{\tX}{\pi^*L}} \cdot  \int_{0}^{\infty} \left( \int_{\Delta_{\ybul|E}(\pi^*L-tE)} \phi_{\ord_E} \right) dt \\ 
		& = & \frac{1}{\vol{\tX}{\pi^*L}} \cdot \int_{0}^{\infty} t\cdot \vol{\tX|E}{\pi^*L-tE}\, dt\ .
	\end{eqnarray*}
\end{proof}

\subsection{Filtered Newton--Okounkov body}

Boucksom--Chen define in \cite[Definition 1.9]{BC} the \emph{filtered Newton--Okounkov body} determined by $ F$ as 
\[\widehat{\Delta}(L, F)\deq\{(\alpha,t)\in\Delta_v(L)\times\R \,|\, 0\le t \le \phi_{ F}(\alpha)\},\]
and it is not hard to see, using Remark \ref{rmk:concave-transforms}, that it is equal to
\begin{equation}\label{eq:slices-BC}
\{(\alpha,t)\in\Delta_v(L)\times\R \,|\, 0\le t, \text{ and }\alpha\in\Delta_v(V_t(\bullet))\}.
\end{equation}In other words, the Boucksom--Chen Newton--Okounkov body is built from its "horizontal" slices, which are, at each level $t$, the Newton--Okounkov body of the graded linear series $V_t(\bullet)$.
Note that this definition throws away the possible regions where the concave transform takes negative values; this makes sense in the arithmetic setting which was the main motivation of \cite{BC}, because then the integral of the positive part of the concave transform turns out to be equal to the arithmetic volume. 
It also makes sense in our case of interest when the filtration comes from $\ord_Z$ which is nonnegative.
However, the integral of the concave transform over the whole Newton--Okounkov body is meaningful as well; at least in the arithmetic toric case it equals the height of $X$ (see \cite{BGMPS} where both integrals are considered).
In general it is worthwile to extend the filtered Newton--Okounkov body towards the negative-$t$ halfspace, replacing the lower bound $0\le t$ in the definition of $\widehat{\Delta}(L, F)$ by $e_{\min}( F)\le t$, to keep all the information encoded by $\phi_{{F}}$.
To fix notation, if $B\in \R$ equals either $0$ or $e_{\min}({{F}})$ we denote
\[\widehat{\Delta}(L, F)_B\deq\{(\alpha,t)\in\Delta_v(L)\times\R \,|\, B\le t \le \phi_{ F}(\alpha)\}.\]

Let now $\Gamma_{{F}}=\langle\Supp  F \rangle_\Z \subset\R$ be the group generated by the support of $ F$ (which we recall need not even be a semigroup itself) and define
\[\Sigma_{v, F,B}\deq\{ (m,x,t) \in \Z^{n+1}\times \Gamma_{ F} \,|\,
 t\ge Bm, x\in v({F}_t(m))\} \subset \R^{n +2}. \]

\begin{proposition}\label{pro:bc-semigroup}
	The semigroup $\Sigma_{v, F,B}$, graded by first component, is linearly bounded and asymptotically convex. 
\end{proposition}
\begin{proof}
Let $(d,x,t)\in \Sigma_{v, F,B}$.
Since ${F}_t(d)\ne 0$, it follows that $B\le t/d\le e_{\max}({F})$.
On the other hand, $x\in v({F}_t(d))\subset v(R_d)$,
so $(d,x)\in \Sigma_v(L)$ which by Proposition \ref{pro:body} is linearly bounded.
So $\Sigma_{v, F,B}$, graded by first component, is linearly bounded.

To see that it is asymptotically convex, let $(d,x,t)\in\langle\Sigma_{v, F,B}\rangle_\Z$ belong to the interior of $\overline{\cone(\Sigma_{v, F,B})}$.
Then $(d,x)$ belongs to the interior, in $\R^{n +1}$, of the image by the projection $(m,y,s)\mapsto(m,y)$ of 
\[\overline{\cone(\Sigma_{v, F,B})} \cap \{(m,y,s) | ds\ge mt\}.\]
Thus there exist $(m_1,y_1,s_1),\dots,(m_r,y_r,s_r)\in \Sigma_{v, F,B}$ with $ds_i\ge m_it$ for each $i$, such that $(d,x)$ belongs to the interior of $\cone((m_1,y_1),\dots,(m_r,y_r))$ and $\langle(m_1,y_1),\dots,(m_r,y_r)\rangle_\Z=\Z^{n+1}$.
Since ${F}$ is a multiplicative filtration, we have 
\[dy_i\in v({F}_{ds_i}(dm_i))\subset v({F}_{m_i t}(dm_i)),\]
so $(dm_i,dy_i,m_i t)\in\Sigma_{v, F,B}$.
Now applying Khovanskii's Theorem \ref{thm:approx_semigroup} to the semigroup 
\[
\Sigma'\deq\langle(dm_1,dy_1),\dots,(dm_r,dy_r)\rangle_\Z\subset\Z^{n +1},
\]
since $(d,x)$ belongs to the interior of $\overline{\cone(\Sigma')}$, it follows that there exist nonnegative integers $a_1,\dots,a_r,$ and $b$ such that 
\[
b(d,x)=a_1(dm_1,dy_1)+\dots+a_r(dm_r,dy_r),
\]
and therefore $a_1m_1t+\dots a_rm_rt=bt$. So 
\[
b(d,x,t)=a_1(dm_1,dy_1,m_1t)+\dots+a_r(dm_r,dy_km_rt),
\]
i.e., $b(d,x,t)\in \Sigma_{v, F,B}$, and this semigroup is asymptotically convex.

\end{proof}

\begin{corollary}
	 $\Delta(\Sigma_{v,{F},B})\cap (\R^{n}\times \{t\})=\Delta_v(V_t(\bullet))$ for every $t\in \langle\Gamma_{{F}}\rangle_\Q$, and $\Delta(\Sigma_{v, F,B})=\widehat \Delta(L,{F})_B$.
\end{corollary}
\begin{proof}
	The equality of the slices of $\Delta(\Sigma_{v, F,B})$ with the Newton--Okounkov bodies of the graded series $V_t(\bullet)$ follows from Theorem \ref{thm:restricted_body_semigrp} applied to $W=\{0\}\times\R^{n}\times\{0\}$, and from this equality it follows that $\Sigma_{v, F,0}=\widehat \Delta(L,{F})$.
\end{proof}

\begin{remark}
	The equality $\Delta(\Sigma_{v, F,0})=\widehat \Delta(L,{F})$ can also be derived observing that the $\Gamma_{{F}}$-rational restricted semigroups of $\Sigma_{v, F,B}$ in the direction $W=\{0\}\times\R^{n}\times\{0\}$ have Newton--Okounkov bodies equal to the slices \eqref{eq:slices-BC}, since we already know from \cite{BC} that these slices form a convex body. 
	However, since our proof does not depend on \cite{BC}, it allows to ``reverse'' the construction of the filtered Newton--Okounkov body: starting from $\Delta(\Sigma_{v,{F},B})$, we have shown that its slices are the Newton--Okounkov bodies of the $V_t(\bullet)$.
	
	Other consequences of Proposition \ref{pro:bc-semigroup} are the equivalence of Remark \ref{rmk:concave-transforms} and the volume formula \cite[Corollary 1.13]{BC}:
\end{remark}

\begin{corollary}
	The two definitions of concave transform of a filtration $F$ as above coincide, i.e., $\phi_{ F}(\alpha)=\phi_{\Finv}(\alpha)$ for all $\alpha\in \Delta_{v}(R)$.
\end{corollary}
\begin{proof}
	By definition $\Delta(\Sigma_{v, F,e_{\min}(F)})$ is the subgraph of $\phi_{F}$.
	By Theorem \ref{thm:restricted_body_semigrp} applied to $W=\{0\}^{n+1}\times\R$,
	it is also the subgraph of $\phi_{\Finv}$.
\end{proof}

\begin{corollary}[Boucksom--Chen {\cite[Corollary 1.13]{BC}}]
	$$\vol{}{\widehat{\Delta}(L,F)_0}=\int_{t=0}^{\infty} \vol{}{\Delta_{v}(V_t(\bullet))}dt
	= \lim_{d\to\infty}\frac{\operatorname{mass}_+(F(d))}{d^{n+1}}.$$
\end{corollary}
\begin{proof}
	Follows from Corollary \ref{cor:volume-and-slices},  applied to $W=\{0\}^{n+1}\times\R$.
\end{proof}

We close this section proving that the filtered Newton--Okounkov body with respect to a filtration of rank 1 is in fact the Newton--Okounkov body of a suitable bounded Rees algebra.

\begin{proposition}
	Let $X$ be an $n$-dimensional projective variety, $v:K(X)^\times \rightarrow\Gamma$ a valuation of maximal rational rank, $L$ a big line bundle on $X$, and ${F}$ a homogeneous, linearly bounded, and complete multiplicative filtration on $R=R(X,L)$ indexed by a subgroup of $\R$. Let $\Rees_B( F)$ the bounded Rees algebra (Definition \ref{def:bounded-rees}) and 
	let $\chi:\R^{n}\times \R \rightarrow \R\times\R^n$ be the isomorphism that exchanges both factors. Then $\chi(\widehat{\Delta}(L,{F})_B)=\Delta_v(\Rees_B({F}))$.
\end{proposition}
\begin{proof}
The semigroup 
\begin{gather*}
\Sigma_v(\Rees_0(F))=\{(d,t,x)\in(\Z\times\Gamma_{F})\times\Gamma \mid
\exists s\in K(X), su_1^d u_2^t\in\Rees_0(F), v(s)=x \}=\\
\{(d,t,x)\in(\Z\times\Gamma_{F})\times\Gamma \mid
\exists s\in K(X), su_1^d \in F_t(d), v(s)=x \}
\end{gather*}
isomorphic to $\Sigma_{v, F,B}$, where the isomorphism switches the factors $\Z^{n}$, coming from the valuation, and $\Gamma_{{F}}$ coming from the filtration.
\end{proof}

\section{Rationality of Seshadri constants on surfaces via integrals of concave transforms}
Let $X$ be a smooth projective variety, $\ord_Z$ a divisorial valuation of $K(X)$, $v$ a valuation of maximal rational rank on $K=K(X)$, and $L$ a big line bundle on $X$. 
This section is devoted to a concrete realization of integrals of concave transforms as volumes, and an application to the rationality of Seshadri constants. 

\subsection{Subgraphs of concave transforms as Newton--Okounkov bodies of line bundles.}
Given $X$,$L$, and ${\rm ord}_Z$ as above, we explicitly construct a big 
line bundle whose volume equals the integral of $\phi_{\ord_Z}$ over any 
Newton--Okounkov body of $L$. 

\begin{theorem}\label{thm:integrals are volumes}
	With notation as above, there exists a projective variety $\xh$, a valuation of maximal rank $\widehat{v}$ on $K(\xh)$ and a big divisor $\lh$ on $\xh$ such that 
	\[
	\Delta_{\widehat{v}} \equ \text{\rm inverted subgraph of $\phi_{{\rm ord}_Z}\colon \Delta_v(L) \lra \RR_{\geq 0}$} \ .
	\]
	In particular, 
	\[
	\int_{\Delta_{v}(L)} \phi_{\ord_Z }\equ \vol{\xh}{\lh}\ .
	\]
\end{theorem}

Here by inverted subgraph of a function $f$ defined on $A$ we mean the set of 
all points
\[ \left\{ (\alpha, x)| 0\leq \alpha\leq f(x)\, ,x\in A\right\}.\]
Note that the integral of $\phi_{\ord_Z}$ is independent of the choice of the valuation $v$. 

\begin{lemma}\label{lem:birational}
	Let $X$ be a normal projective variety, $L$ a big divisor on $L$, $ord_Z$ a divisorial valuation of the function field $K(X)$, and $\pi\colon X'\to X$ a proper
	birational morphism. Then
	
	\[
	\int_{\Delta_{v_1}} \phi_{\ord_Z} \equ \int_{\Delta_{v_2}}(\pi^*L) \phi_{\ord_Z}\ ,
	\]
	where $v_1$ and $v_2$ are arbitrary valuations of maximal rational rank on $K(X)=K(X')$.
\end{lemma}
\begin{proof}
We know that these two integrals are independent of the choice of valuations $v_i$ because they are volumes of $\Delta_{v_i}(\Rees_0(F_{\ord_Z}))$, and by Theorem \ref{thm:Hilbert-function-algebra-semigroup} these only depend on the Hilbert function of $\Rees_0(F_{\ord_Z})$.
It will therefore be enough to find one example of flag valuations $v_{\ybul}$ and $v_{\ybul'}$ for which the two integrals coincide.\\ \\
Let $\ybul$ be a flag on $X$ such that the point $Y_n$ is contained in the open set $U$ over which $\pi$ is an isomorphism and let $\ybul'$ be the proper
transform of $\ybul$ on $X'$. Since $X$ is normal the pullback map
\[ \pi^*: H^0(dL)\rightarrow H^0(d\pi^*(L))\] 
is an isomorphism for every $d$. Since the valuations $v_{\ybul}$ and $v_{\ybul'}$ can be calculated over the isomorphic open sets $U$ and $\pi^{-1}(U)$ we have that 
for any $\sigma \in H^0(dL)$ \[v_{\ybul'}(\pi^*(\sigma))= v_{\ybul}(\sigma).\]
It follows that for this choice of $\ybul$ and $\ybul'$  we have that 
$\Delta_{\ybul}(L)=\Delta_{\ybul'}(\pi^*(L))$ and
$\phi_{\ord_Z}= \phi_{\ord_Z}$. This completes the proof of Lemma \ref{lem:birational}
\end{proof}

By Lemma~\ref{lem:birational} we may assume after possibly blowing up $X$ that $\ord_Z= \ord_D$, with $D$ a smooth effective Cartier divisor on $X$.

We set 
\[
\xh \deq \PP_X(\sO_X\oplus\sO_X(D))\ .
\]
where here we have used the Grothendieck convention for projective bundles. 
The natural surjections $\sO_X\oplus\sO_X(D)\to \sO_X$ and $\sO_X\oplus\sO_X(D)\to \sO_X(D)$ give rise to embeddings $\iota_1\colon X\hookrightarrow \xh$ and 
$\iota_2\colon X\hookrightarrow \xh$, whose respective images we will denote by $X_1$ and $X_2$. Note that $X_1\cap X_2=\emptyset$. 

In addition we have the natural projection $\pi\colon \xh\to X$ whose restriction to the $X_i$'s is the identity of $X_i\simeq X$.  The construction also gives rise to  the 
linear equivalence $X_2\sim X_1 + \pi^*D$, and we have isomorphisms
\[
\sO_{\xh}(X_1)|_{X_1} \simeq \sO_X(-D)\ \ \text{ and }\ \ \sO_{\xh}(X_2)|_{X_2} \simeq \sO_{X_2}(D)\ . 
\]
We will set $\lh \deq \pi^*L+bX_1$, for some rational number $b$ such that $b> \sup \{s>0\mid  L-sD\mbox{ is big  }\}$. We consider the composite valuation (Example \ref{exa:composite}) obtained from $Y_1=X_2$, $\bar{v}=v$, which we denote $\widehat{v}$.

We denote the subgraph of the function $\phi_{\ord_D}$ on $\Delta_{v}(L)$ by $\dht$, i.e.
\[
\dht \deq \st{(\alpha, (t_1,\dots,t_n))\mid (t_1,\dots,t_n)\in \delta_v(L)\, ,\, 0\leq \alpha\leq \phi_{\ord_D}(t_1,\dots,t_n)} \dsubseteq \RR^n\times\RR\ .
\]

\begin{proof}[Proof of Theorem~\ref{thm:integrals are volumes}]
It will be enough to prove that 
\[
	\dht \equ \Delta_{\widehat v}(\widehat{L})\ ,
\]
By definition we have that 
\begin{eqnarray*}
\dht & = &  \text{topological closure of}\  \bigcup_{d=1}^{\infty} \st{(\alpha,
(\frac{1}{d}\cdot v(s))) \mid s\in \HH{0}{X}{\sOX(dL)}\, ,\, \ord_D(s)\geq \alpha} \dsubseteq \QQ^n\times \QQ\\
& = &   \text{topological closure of}\  \bigcup_{d=1}^{\infty} \st{(\alpha,
\frac{1}{d}\cdot v(s)) \mid s\in \HH{0}{X}{\sOX(dL-\alpha D)}} \dsubseteq \QQ^n\times \QQ  \ .
\end{eqnarray*}
Let us write 
\[
S_1(d) \deq \st{(\frac{1}{d}\cdot v(s),\alpha) \mid s\in \HH{0}{X}{\sOX(d(L-\alpha D))}} \ .
\]
Next, we look at  the convex body $\Delta_{\widehat{v}}(\widehat L)$. By definition
\[
\Delta_{\widehat{v}}(\widehat L) \equ  \text{topological closure of }\ \bigcup_{d=1}^{\infty} \st{\frac{1}{d}\cdot \widehat{v}(\sigh)\mid \sigh\in\HH{0}{\xh}{\sO_{\xh}(d\lh)}} \ .
\]
We set
\[ S_2(d)= \st{\frac{1}{d}\cdot \widehat{v}(\sigh)\mid \sigh\in\HH{0}{\xh}{\sO_{\xh}(d\lh)}}\]
By the construction of the composite valuation $\widehat v$ we have that 
$\widehat{v}(\sigh) \equ (\ord_{X_2}\sigh, v(\sigh_1))$, where 
\[\sigh_1 \deq \frac{\sigh}{f^{\ord_{X_2}\sigh}}\big|_{X_2},\]  
the function $f$ begin a local equation of $X_2$ in $\xh$ in a neighbourhood of
$\yh_{n+1}$. It follows that
\[S_2(d)=
\st{\frac{1}{d}\cdot \widehat{v}(\sigh)\mid \sigh\in\HH{0}{\xh}{\sO_{\xh}(d\lh)}} \equ  
\st{ (\alpha,\frac{1}{d}\cdot v(\sigh_1)) \mid \sigh\in \HH{0}{\xh}{\sO_{\xh}(d\lh)}\, ,\, \ord_{X_2}\sigh = d\alpha}\ .
\]
We have a  natural injection $j_2\colon \HH{0}{\xh}{\sO_{\xh}(d(\lh-\alpha X_2))} \hookrightarrow \HH{0}{\xh}{\sO_{\xh}(d\lh)}$ and for any global section
$\sigh\in j_2\colon \HH{0}{\xh}{\sO_{\xh}(d(\lh-\alpha X_2))}$ we have that
 $\sigh_1=j_2^{-1}(\sigh)|_{X_2}$. It follows that 
\[ 
S_2(d) \deq  \left\{ (\alpha, \frac{1}{d}\cdot v(\sigh|_{X_2})) \mid \sigh\in \HH{0}{\xh}{d(\lh-\alpha X_2)} \right\} \ .
\]

We will now prove that $S_1(d)=S_2(d)$ are the same set, which completes the
proof of the theorem.
\begin{eqnarray*}
(d(\lh-\alpha X_2))|_{X_2} & \equ &  \iota_2^*(d(\lh-\alpha X_2)) \equ \iota^*_{2} (d(\pi^*L+bX_1 -\alpha X_2))  \\
&  \equ  &   (\iota_2^*\pi^*)(dL)  + \iota_2^*(dbX_1) - \iota_2^*(d\alpha X_2)    \\
& = & d(L-\alpha D)\ , 
\end{eqnarray*}
since $\iota_2\circ\pi=\id_X$, $\iota_2^*\sO_{\xh}(X_1)=\sO_{X}$, and $\iota_2^*\sO_{\xh}(X_2)=\sO_{X_2}(D)$ by construction.  
We will therefore have that $S_1(d)=S_2(d)$ for all $d\geq 1$ if the 
restriction map 
\[ 
H^0(\xh,\sO_{\xh}(d(\lh-\alpha X_2)))\stackrel{\res_{X_2}}{\lra }  H^0(X_2,\sO_{X_2}(d(L-\alpha D)))
\]
is surjective 
for any rational $a$ and integral $d$ such that $d\alpha D$ is an 
integral divisor.  
	
This is immediate for any $\alpha$ for which $L-\alpha D$ is not effective, so we may assume that $\alpha \leq \mu(L;D)$.
Observe that  
\[ 
d(\lh-\alpha X_2) \equ  d(\pi^*L+ bX_1-\alpha X_2) \equ  d(\pi^*(L-\alpha D)+ (b-\alpha) X_1)\ ,
\]
and since we have chosen $b> \mu(L;D)$,  it follows that $(b-\alpha)X_1$ is effective. Since  $X_1\cap X_2=\varnothing$ and 
$H^0(\pi^*(L-\alpha D)) = \pi^*(H^0(L-\alpha D))$ we deduce from the commutative diagram 
\[
\xymatrix{
	H^0(\xh,\OO_{\xh}(d(\lh-\alpha X_2))) \ar[rr]^{\res_{X_2}} & & H^0(X_2,\OO_{X_2}(d(L-\alpha D))) \\ 
	H^0(X,\OO_X(d(L-\alpha D))) \ar[u]_{\pi^*} \ar[urr]^{\sim} & 
	}
\]
that 
\[ 
H^0(d(\lh-\alpha X_2))\twoheadrightarrow  H^0(d(L-\alpha D))\ ,
\]
so $S_1(d)$ and $S_2(d)$ are equal. This completes the proof of the theorem.
\end{proof}

\subsection{Link with Seshadri constants.}
We start by defining Seshadri constants.  
Let $X$ be a smooth projective variety, $L$ an ample line bundle, $x\in X$ an
arbitrary point. We denote the blow-up of $x\in X$ with exceptional
divisor $E$  by $\pi:\tilde{X}\to X$.
\\ \\
The Seshadri constant of $L$ at $x$ is defined as
\[
 \epsilon(L;x) \deq \sup\st{t>0\,|\, \pi^*L-tE \text{ is nef }}\ .
\]
We also consider the invariant
\[
 \mu(L;x) \deq \sup\st{t>0\,|\, \pi^*L-tE \text{ is pseudo-effective}} \equ \sup\st{t>0\,|\, \pi^*L-tE \text{ is big}} \ .
\]
 Note that if $\ybul$ is a flag on $\tilde{X}$ whose first member is $E$
and $\Delta$ is the Newton--Okounkov body of $L$ with respect to this flag then
the projection of $\Delta$ onto its first coordinate is an interval of the form
\[[\beta(L,p), \mu(L,p)].\]
This invariant is sometimes denoted by $\mu_E(\pi^*L)$ or $\mu(\pi^*L,E)$. 
\\ \\
The 
basic link between rationality of Seshadri constants on surfaces 
and the invariant $\mu$
is the following, taken 
from  \cite[Remark 2.3]{KLM}.
\begin{remark}\label{rmk:KLM}
Let $X$ be a smooth projective surface and let $L$ be a line bundle and $x$ a 
point on $X$. 
If $\epsilon(L;p)$ is irrational, then
\[
 \epsilon(L;x) \equ \mu(L;x)\ .
\]
In particular, if $\mu(L;x)$ is rational, then so is $\epsilon(L;x)$.
\end{remark}
We now link this rationality to that of a third invariant, the integral of the
concave transform. For any variety $X$, any  big line bundle $L$ and any point
$x\in X$ we define 
\[\ilx= \int_{\Delta_{\ybul}(L)} \phi_{v_x}\]
where $v_p$ is the order at  $p$ valuation. 
We have the following proposition. 

\begin{proposition}\label{prop:integral estimate}
Let $X$ be a smooth projective variety of dimension $n$, 
$L$ an ample Cartier divisor on $X$, $x\in X$ arbitrary. Then 
\[
\ilx \dgeq  \frac{\elx^{n+1}}{(n+1)!\, (L^n)} 
\]
with equality if  $\elx=\mlx$. 
\end{proposition}

\begin{proof}
We consider a flag $\ybul$ on $\tilde{X}$ whose first member is $E$, and denote by $\Delta$ the associated Newton--Okounkov body $\Delta_{\ybul}(L)$. 
By definition, the function $\phi_{v_p}$ on this body is given by
\[ \phi_{v_p}(t_1,\ldots,t_n)= t_1.\]
By Fubini's theorem we therefore have that 
\[ \ilx=\int_\Delta t_1= \int_0^{\mu(L,x)} s {\rm vol}(\Delta\cap(s\times \mathbb{R}^{n-1}) ds\]
and applying \cite{LM09}, Lemma 6.3 which states that
\[ {\rm vol}(\Delta\cap(s\times \mathbb{R}^{n-1})= \vol{\tX|E}{f^*L-tE}\]
we get that
\begin{eqnarray*}
\ilx  & \equ &  \frac{1}{\vol{X}{L}} \int_{0}^{\elx} t\cdot \vol{\tX|E}{f^*L-tE} dt \ +\ \frac{1}{\vol{X}{L}} \int_{\elx}^{\mlx} t\cdot \vol{\tX|E}{f^*L-tE} dt \\ 
& \dgeq & \frac{1}{\vol{X}{L}} \int_{0}^{\elx} t\cdot \vol{\tX|E}{f^*L-tE} dt \ ,
\end{eqnarray*}
with equality if $\elx=\mlx$. We will determine the expression on the right. Since $L$ is ample, $\vol{X}{L} = (L^n)$. By definition of 
$\epsilon(L,x)$  the divisor $f^*L-tE$ is ample if $0<t<\elx$, hence 
\[
\vol{\tX|E}{f^*L-tE} \equ \vol{E}{f^*L-tE|_E} \equ \vol{\PP^{n-1}}{f^*L-tE|_E} \equ t^{n-1}\ .
\]  
Consequently,
\[
\frac{1}{\vol{X}{L}} \int_{0}^{\elx} t\cdot \vol{\tX|E}{f^*L-tE} dt \equ \frac{1}{(L^n)} \int_{0}^{\elx} t^n dt \equ \frac{\elx^{n+1}}{(n+1)!\, (L^n)}\ ,
\]
which is what we wanted. 
\end{proof}

\begin{corollary}\label{cor:irrational Seshadri on surfaces}
Let $X$ be a smooth projective surface, $x\in X$, and $L$ an ample Cartier divisor on $X$. Then $\elx$ is rational if $\ilx$ is. 
\end{corollary}
\begin{proof}
Suppose  that $\elx$ is irrational. Then necessarily $\elx=\mlx$ and $\elx=\sqrt{(L^2)}$. From Proposition~\ref{prop:integral estimate} it follows that 
$\ilx=\elx/(n+1)!$ so  $\ilx$ is also irrational.
\end{proof} 

We will now apply  Theorem~\ref{thm:integrals are volumes} to calculating 
$\ilx$. Let $X$ be  a smooth projective surface, let $x$ be a point in $X$, 
 and let $L$ be an ample Cartier  on $X$. Let $\eta\colon \tX\to X$ be the blowing-up of $X$ at the point $x\in X$ with exceptional divisor $E$, write 
\[
\pi\colon \xh \deq \PP_{\tX}(\OO_{\tX}\oplus\OO_{\tX}(E))) \lra \tX\ 
\]
for the natural projection, and set $f\deq \pi\circ\eta $. We consider sub varieties $X_1$ and $X_2$ of $X$ as in the proof of 
Theorem~\ref{thm:integrals are volumes}. We then have that 
\[ 
\ilx \equ  \vol{\xh}{\lh}
\]
where $\lh$ is a line bundle on $\xh$ of the form 
\[
\lh \equ f^*L + bX_1 \equ \pi^*(\eta^*L) + b(X_2-\pi^*E)\ ,
\]
for some $b> \mu(L,x)$.
Since $X_2 = \xi$, where $\OO_{\xh}(\xi)=\OO_{\xh}(1)$ we obtain on rearranging 
\[
\lh \equ \pi^*(\eta^*L) + b(\xi-\pi^*E) \equ b\xi + \pi^*(\eta^*L-bE)\ .
\]	
Corollary~\ref{cor:irrational Seshadri on surfaces} then gives us the 
following
\begin{corollary}\label{cor:fg implies Sehsadri rtl}
With notation as above, if $\vol{\xh}{\lh}\in \QQ$ (in particular, if the section ring $R(\xh,\lh)$ is finitely generated), then $\elx$ is a rational number. 
\end{corollary}

\begin{remark}\label{rmk:nef cone of projective bundle}
The nef cone of  $\xh$ equals the closed convex subcone of  $N^1(\xh)_\RR$ generated by the classes $\pi^*\Nef(\tX)$ and the classes $\xi+\pi^*H$ such that both $H$ and $H+E$ are nef on $\tX$.  Recalling that $b>\mlx \geq \elx$ needs to be satisfied, we see  $\lh$ cannot be ample since $\eta^*L-bE$ never is. 
\end{remark} 

By \cite{KKM}  (see also \cite{BCHM,CL})
we know that $R(\xh,\lh)$ is finitely generated whenever 
$(\lh-K_{\xh})$ is big and nef. Since  
\[
K_{\xh} \equ - 2\xi +  \pi^*(K_{\tX} + \det (\OO\oplus\OO(E))) \equ -2\xi + \pi^*(\eta^*K_X+2E)\ ,  
\]
this amounts to verifying that  
\[
\lh - K_{\xh}  \equ  (b+2)\xi + \pi^*( \eta^*(L-K_X) - (b+2)E) 
\]
is big and nef. 

\begin{corollary}\label{cor:rationality of Seshadri}
With notation as above, assume that there exists a positive integer $b$ satsifying $\elx < b < \epsilon(L-K_X;x) -2$. Then $\elx$ is rational. 
\end{corollary} 
\begin{proof}
To begin with note that the condition includes the assumptions that both $L$ and $L-K_X$ are ample. 
If $\elx<\mlx$ then $\elx$ is automatically rational, therefore we can assume 
$\elx=\mlx$. We then have that $\ilx= {\rm vol}(\lh)$ with since $b> \mlx$ by  	
Theorem~\ref{thm:integrals are volumes}. It remains only to check that
\[
\lh-K_{\xh} \equ (b+2)\xi + \pi^*( \eta^*(L-K_X) - (b+2)E) \equ (b+2)\left(\xi + \frac{1}{b+2}\pi^*(\eta^*(L-K_X)-(b+2)E)\right)
\]
is big and nef, which,  by Remark~\ref{rmk:nef cone of projective bundle} and \cite[Lemma 2.3.2]{Laz04I} is certainly implied if 
$\eta^*(L-K_X) - (b+2)E$  and   $\eta^*(L-K_X) - (b+2)E + (b+2) E= \eta^*(L-K_X)$ are both big and nef.

By definition of $\epsilon(L- K_X, x)$, the condition $b+2< \epsilon(L-K_X,x)$ implies that the former is ample and the latter is big and nef. 
But then $R(\xh,\lh)$ is finitely generated by \cite{KKM} (see also \cite{BCHM,CL}), therefore $\elx\in\QQ$ by Corollary~\ref{cor:fg implies Sehsadri rtl}. 
\end{proof}

An immediate consequence is the rationality of Seshadri constants on surfaces with positive anticanonical class. The result below is not new, however, we obtain it without any specific knowledge about negative curves on the blow-up of $X$. 

\begin{remark}
Keeping the notation assume that $\epsilon(-K_X) \geq 3$. Then
\[
\epsilon(L-K_X;x) -2 \dgeq \elx + \epsilon(-K_X;x) + 2 \,>\, \elx +1\ ,
\]
hence there will exist an integer $b$ as in Corollary~\ref{cor:rationality of Seshadri}. Consequently, $\elx\in\QQ$. 
\end{remark}

% % % % % % % Bibliography % % % % % % % %

{\footnotesize
	\bibliographystyle{plain}
	\bibliography{NOB}{}}

\begin{thebibliography}{10}

\bibitem{Abh56}
Shreeram Abhyankar.
\newblock On the valuations centered in a local domain.
\newblock {\em Amer. J. Math.}, 78:321--348, 1956.

\bibitem{AKL}
Dave Anderson, Alex K\"uronya, and Victor Lozovanu.
\newblock Okounkov bodies of finitely generated divisors.
\newblock {\em Int. Math. Res. Not. IMRN}, (9):2343--2355, 2014.

\bibitem{BCHM}
Caucher Birkar, Paolo Cascini, Christopher~D. Hacon, and James McKernan.
\newblock Existence of minimal models for varieties of log general type.
\newblock {\em J. Amer. Math. Soc.}, 23(2):405--468, 2010.

\bibitem{Bou14}
S\'ebastien Boucksom.
\newblock Corps d'{O}kounkov (d'apr\`es {O}kounkov, {L}azarsfeld-{M}usta\c t\v
  a et {K}aveh-{K}hovanskii).
\newblock {\em Ast\'erisque}, (361):Exp. No. 1059, vii, 1--41, 2014.

\bibitem{BC}
S\'ebastien Boucksom and Huayi Chen.
\newblock Okounkov bodies of filtered linear series.
\newblock {\em Compos. Math.}, 147(4):1205--1229, 2011.

\bibitem{BKMS}
S{\'e}bastien Boucksom, Alex K{\"u}ronya, Catriona Maclean, and Tomasz
  Szemberg.
\newblock Vanishing sequences and {O}kounkov bodies.
\newblock {\em Math. Ann.}, 361(3-4):811--834, 2015.

\bibitem{Bou06}
Nicolas {Bourbaki}.
\newblock {\em {\'El\'ements de math\'ematique. Alg\`ebre commutative.
  Chapitres 5 \`a 7.}}
\newblock Berlin: Springer, reprint of the 1985 original edition, 2006.

\bibitem{BGMPS}
Jos\'{e}~Ignacio Burgos~Gil, Atsushi Moriwaki, Patrice Philippon, and
  Mart\'{i}n Sombra.
\newblock Arithmetic positivity on toric varieties.
\newblock {\em J. Algebraic Geom.}, 25(2):201--272, 2016.

\bibitem{CL}
Paolo Cascini and Vladimir Lazi\'{c}.
\newblock New outlook on the minimal model program, {I}.
\newblock {\em Duke Math. J.}, 161(12):2415--2467, 2012.

\bibitem{CHPW}
Sung~Rak {Choi}, Yoonsuk {Hyun}, Jinhyung {Park}, and Joonyeong {Won}.
\newblock {Okounkov bodies associated to pseudoeffective divisors}.
\newblock {\em arXiv e-prints}, page arXiv:1508.03922, August 2015.

\bibitem{CHPW2}
Sung~Rak {Choi}, Jinhyung {Park}, and Joonyeong {Won}.
\newblock {Okounkov bodies associated to pseudoeffective divisors II}.
\newblock {\em arXiv e-prints}, page arXiv:1608.00221, July 2016.

\bibitem{CFKLRS}
Ciro Ciliberto, Michal Farnik, Alex K\"uronya, Victor Lozovanu, Joaquim Ro\'e,
  and Constantin Shramov.
\newblock Newton-{O}kounkov bodies sprouting on the valuative tree.
\newblock {\em Rend. Circ. Mat. Palermo (2)}, 66(2):161--194, 2017.

\bibitem{Cut}
S.~Dale Cutkosky.
\newblock Zariski decomposition of divisors on algebraic varieties.
\newblock {\em Duke Math. J.}, 53(1):149--156, 1986.

\bibitem{Dem92}
J.P. Demailly.
\newblock Singular {H}ermitian metrics on positive line bundles.
\newblock In K~Hulek et~al., editors, {\em Complex Algebraic Varieties
  (Bayreuth 1990)}, volume 1507 of {\em LNM}, pages 87--104. Springer, 1992.

\bibitem{Don}
S.~K. Donaldson.
\newblock Scalar curvature and stability of toric varieties.
\newblock {\em J. Differential Geom.}, 62(2):289--349, 2002.

\bibitem{DKMS}
M.~Dumnicki, A.~K\"uronya, C.~Maclean, and T.~Szemberg.
\newblock Rationality of {S}eshadri constants and the
  {S}egre-{H}arbourne-{G}imigliano-{H}irschowitz conjecture.
\newblock {\em Adv. Math.}, 303:1162--1170, 2016.

\bibitem{DKMS2}
M.~Dumnicki, A.~K\"{u}ronya, C.~Maclean, and T.~Szemberg.
\newblock Seshadri constants via functions on {N}ewton-{O}kounkov bodies.
\newblock {\em Math. Nachr.}, 289(17-18):2173--2177, 2016.

\bibitem{FFL}
Xin {Fang}, Ghislain {Fourier}, and Peter {Littelmann}.
\newblock {On toric degenerations of flag varieties}.
\newblock {\em arXiv e-prints}, page arXiv:1609.01166, September 2016.

\bibitem{Fekete}
M.~Fekete.
\newblock \"{U}ber die {V}erteilung der {W}urzeln bei gewissen algebraischen
  {G}leichungen mit ganzzahligen {K}oeffizienten.
\newblock {\em Math. Z.}, 17(1):228--249, 1923.

\bibitem{FS01}
L\'{a}szl\'{o} Fuchs and Luigi Salce.
\newblock {\em Modules over non-{N}oetherian domains}, volume~84 of {\em
  Mathematical Surveys and Monographs}.
\newblock American Mathematical Society, Providence, RI, 2001.

\bibitem{HK}
Megumi Harada and Kiumars Kaveh.
\newblock Integrable systems, toric degenerations and newton--okounkov bodies.
\newblock {\em Inventiones Math.}, 202(3):927--985, 2015.

\bibitem{HAG}
Robin Hartshorne.
\newblock {\em Algebraic geometry}.
\newblock Springer-Verlag, New York-Heidelberg, 1977.
\newblock Graduate Texts in Mathematics, No. 52.

\bibitem{Jow}
Shin-Yao Jow.
\newblock Okounkov bodies and restricted volumes along very general curves.
\newblock {\em Adv. Math.}, 223(4):1356--1371, 2010.

\bibitem{KK12}
Kiumars Kaveh and A.~G. Khovanskii.
\newblock Newton-{O}kounkov bodies, semigroups of integral points, graded
  algebras and intersection theory.
\newblock {\em Ann. of Math. (2)}, 176(2):925--978, 2012.

\bibitem{KKM}
Sean Keel, Kenji Matsuki, and James McKernan.
\newblock Log abundance theorem for threefolds.
\newblock {\em Duke Math. J.}, 75(1):99--119, 1994.

\bibitem{Kho95}
A.~G. Khovanskii.
\newblock Sums of finite sets, orbits of commutative semigroups and {H}ilbert
  functions.
\newblock {\em Funktsional. Anal. i Prilozhen.}, 29(2):36--50, 95, 1995.

\bibitem{KL_Reider}
Alex {K{\"u}ronya} and Victor {Lozovanu}.
\newblock {A Reider-type theorem for higher syzygies on abelian surfaces}.
\newblock {\em arXiv e-prints}, page arXiv:1509.08621, September 2015.

\bibitem{KL_inf}
Alex K\"uronya and Victor Lozovanu.
\newblock Infinitesimal {N}ewton-{O}kounkov bodies and jet separation.
\newblock {\em Duke Math. J.}, 166(7):1349--1376, 2017.

\bibitem{KL_noninf}
Alex K\"{u}ronya and Victor Lozovanu.
\newblock Positivity of line bundles and {N}ewton-{O}kounkov bodies.
\newblock {\em Doc. Math.}, 22:1285--1302, 2017.

\bibitem{KL_Geom}
Alex K\"uronya and Victor Lozovanu.
\newblock Geometric aspects of newton--okounkov bodies.
\newblock In J.~Buczynski, S.~Cynk, and T.~Szemberg, editors, {\em
  Phenomenological approach to algebraic geometry}, volume 116 of {\em Banach
  Center Publications}. Polish Academy of Sciences, 2018.

\bibitem{KLM}
Alex K\"uronya, Victor Lozovanu, and Catriona Maclean.
\newblock Convex bodies appearing as {O}kounkov bodies of divisors.
\newblock {\em Adv. Math.}, 229(5):2622--2639, 2012.

\bibitem{KLM_volume}
Alex K\"uronya, Victor Lozovanu, and Catriona Maclean.
\newblock Volume functions of linear series.
\newblock {\em Math. Ann.}, 356(2):635--652, 2013.

\bibitem{KMSwB}
Alex K\"uronya, Catriona Maclean, and Tomasz Szemberg.
\newblock Functions on okounkov bodies coming from geometric valuations (with
  an appendix by s\'ebastien boucksom).

\bibitem{Laz04I}
Robert Lazarsfeld.
\newblock {\em Positivity in algebraic geometry. {I}}, volume~48 of {\em
  Ergebnisse der Mathematik und ihrer Grenzgebiete. 3. Folge. A Series of
  Modern Surveys in Mathematics [Results in Mathematics and Related Areas. 3rd
  Series. A Series of Modern Surveys in Mathematics]}.
\newblock Springer-Verlag, Berlin, 2004.
\newblock Classical setting: line bundles and linear series.

\bibitem{LM09}
Robert Lazarsfeld and Mircea Musta\c{t}\u{a}.
\newblock Convex bodies associated to linear series.
\newblock {\em Ann. Sci. \'Ec. Norm. Sup\'er. (4)}, 42(5):783--835, 2009.

\bibitem{Ok1}
Andrei Okounkov.
\newblock Brunn-{M}inkowski inequality for multiplicities.
\newblock {\em Invent. Math.}, 125(3):405--411, 1996.

\bibitem{Ok2}
Andrei Okounkov.
\newblock Why would multiplicities be log-concave?
\newblock In {\em The orbit method in geometry and physics ({M}arseille,
  2000)}, volume 213 of {\em Progr. Math.}, pages 329--347. Birkh\"{a}user
  Boston, Boston, MA, 2003.

\bibitem{PU}
Elisa {Postinghel} and Stefano {Urbinati}.
\newblock {Newton-Okounkov bodies and Toric Degenerations of Mori dream spaces
  via Tropical compactifications}.
\newblock {\em arXiv e-prints}, page arXiv:1612.03861, December 2016.

\bibitem{RW}
Konstanze {Rietsch} and Lauren {Williams}.
\newblock {Newton-Okounkov bodies, cluster duality, and mirror symmetry for
  Grassmannians}.
\newblock {\em arXiv e-prints}, page arXiv:1712.00447, November 2017.

\bibitem{Roc70}
R.~Tyrrell Rockafellar.
\newblock {\em Convex analysis}.
\newblock Princeton Landmarks in Mathematics. Princeton University Press,
  Princeton, NJ, 1997.
\newblock Reprint of the 1970 original, Princeton Paperbacks.

\bibitem{Roe}
Joaquim Ro\'{e}.
\newblock Local positivity in terms of {N}ewton-{O}kounkov bodies.
\newblock {\em Adv. Math.}, 301:486--498, 2016.

\bibitem{RT1}
Julius Ross and Richard Thomas.
\newblock An obstruction to the existence of constant scalar curvature
  {K}\"{a}hler metrics.
\newblock {\em J. Differential Geom.}, 72(3):429--466, 2006.

\bibitem{RT2}
Julius Ross and Richard Thomas.
\newblock A study of the {H}ilbert-{M}umford criterion for the stability of
  projective varieties.
\newblock {\em J. Algebraic Geom.}, 16(2):201--255, 2007.

\bibitem{Szek}
G\'{a}bor Sz\'{e}kelyhidi.
\newblock Filtrations and test-configurations.
\newblock {\em Math. Ann.}, 362(1-2):451--484, 2015.
\newblock With an appendix by S\'{e}bastien Boucksom.

\bibitem{WN}
David Witt~Nystr\"om.
\newblock Test configurations and {O}kounkov bodies.
\newblock {\em Compos. Math.}, 148(6):1736--1756, 2012.

\bibitem{WN14}
David Witt~Nystr\"{o}m.
\newblock Transforming metrics on a line bundle to the {O}kounkov body.
\newblock {\em Ann. Sci. \'{E}c. Norm. Sup\'{e}r. (4)}, 47(6):1111--1161, 2014.

\bibitem{ZS75II}
O.~Zariski and P.~Samuel.
\newblock {\em Commutative algebra. {V}ol. {II}}.
\newblock Springer-Verlag, New York, 1975.
\newblock Reprint of the 1960 edition, Graduate Texts in Mathematics, Vol. 29.

\end{thebibliography}

\end{document}